\documentclass[11pt,thmsa]{amsart}

\usepackage{amssymb,latexsym}
\usepackage{amsmath,amscd}
\usepackage[pdftex,all]{xy}
\usepackage{color}
\usepackage{hyperref}
\usepackage{esvect}
\usepackage{framed}
\usepackage{comment}
\usepackage{amsmath}
\usepackage{amsthm}
\usepackage{amsfonts}
\usepackage{amssymb}
\usepackage[pdftex]{graphicx}
\usepackage{wrapfig}
\usepackage{layout}
\usepackage{verbatim}
\usepackage{alltt}
\usepackage{ mathrsfs }
\usepackage{tabularx}

\newtheorem{theorem}{Theorem}[section]
\newtheorem{lemma}[theorem]{Lemma}
\newtheorem{proposition}[theorem]{Proposition}

\newtheorem{corollary}[theorem]{Corollary}
\theoremstyle{definition}
\newtheorem{definition}[theorem]{Definition}
\newtheorem{example}[theorem]{Example}
\newtheorem{remark}[theorem]{Remark}

\newtheorem*{ack}{Acknowledgements}

\newcommand{\bS}{\overline{S}}

\newcommand{\de}{\partial}

\newcommand{\sL}{\mathcal{L}}
\newcommand{\sM}{\mathcal{M}}

\newcommand{\Hom}{\operatorname{Hom}}

\newcommand{\ten}{\otimes}

\newcommand{\op}{\operatorname}

\newcommand{\K}{\mathbb{K}\,}
\newcommand\C{\mathbb{C}}

\newcommand\Del{\operatorname{Del}}
\newcommand\Tot{\operatorname{Tot}}

\newcommand{\rh}{\rightarrow}

\newcommand{\solose}{\Rightarrow}

\newcommand{\MC}{\operatorname{MC}}
\newcommand{\Coder}{\operatorname{Coder}}

\newcommand{\id}{\operatorname{id}}

\newcommand{\bs}{\overline{S}}

\def\aaltbin#1#2{\ensuremath{\left(\kern-.35em\left(\genfrac{}{}{0pt}{}{#1}{#2}\right)\kern-.35em\right)} }

\begin{document}
	
	\title[Descent of Deligne-Getzler $\infty$-groupoids]{Descent of Deligne-Getzler $\infty$-groupoids}

\author{Ruggero Bandiera}
\address{Universit\`{a} degli studi di Roma "La Sapienza"}
\email{bandiera@mat.uniroma1.it}
	
\begin{abstract} We prove that Getzler's higher generalization of the Deligne groupoid commutes with totalization and homotopy limits.
\end{abstract}
\maketitle

\section*{Introduction}

It is nowadays well understood that every reasonable deformation problem over a field of characteristic zero is controlled by some differential graded (dg) Lie algebra, or more in general by an $L_\infty$ algebra. This important principle (which, to the best of our knowledge, first appeared in the paper \cite{Nij-Ric}) was sponsored by Deligne, Drinfeld and others over the eighties as a philosophy, and has been  recently made into a rigorous theorem by Lurie and Pridham \cite{lurie2,pridham}, working in the context of derived geometry (a first important step in this direction was made by Manetti \cite{EDF}). In this approach to deformation theory, a fundamental role is played by the Deligne groupoid of a dg Lie algebra \cite{GoMil1}.

In order to associate a controlling dg Lie algebra to a given deformation problem, there are several standard tools available. Among these, Hinich's theorem on descent of Deligne groupoids \cite{Hinichdescent} is especially useful. It applies when we want to deform some global structure on a space $X$, and we have a sheaf $\sL$ of \emph{non-negatively graded} dg Lie algebras over $X$, controlling the deformation problem locally: in this situation, Hinich's theorem tells us that a certain dg Lie algebra representing the \v{C}ech complex $C^*(X;\sL)$ controls the global defomation problem. For more details, and for concrete applications of this theorem to deformation theory, we refer to the original paper \cite{Hinichdescent}, as well as \cite{BM,FMMscdgla,FioMarIac,iacMan2}. 

Hinich's theorem fails when the involved dg Lie algebras are non-trivial in negative degrees. The reason is that in this case the Deligne groupoid is not the right object to consider anymore: in fact, it's the truncation of a more fundamental higher groupoid. For instance, when the dg Lie algebras are concentrated in degrees $\geq-1$, the right object to consider is the associated Deligne 2-groupoid \cite{GetzlerDel}. A proof of the corresponding theorem on descent of Deligne 2-groupoids, together with applications to deformation quantization, can be found in the references \cite{gerbes,yeku2}. 

In the general situation, the object to consider is a certain $\infty$-groupoid associated to the dg Lie algebra $L$. A first model for this $\infty$-groupoid was introduced by Sullivan \cite{sullivan} and studied in depth by Hinich \cite{Hinichdescent,hinichdgC}: it is the Kan complex $\MC_\infty(L):=\MC(\Omega^*(\Delta_\bullet;L))$ of Maurer-Cartan forms on the standard cosimplicial simplex $\Delta_\bullet$ with coefficients in  $L$. A second model, which is the subject of this paper, was introduced by Getzler \cite{GetzlerLie}: we shall denote it by  $\Del_\infty(L)$. While the two models are homotopy equivalent as Kan complexes, Getzler's model is smaller and contains more algebraic information, such as the Baker-Campbell-Hausdorff product on $L^0$ and the Gauge action on Maurer-Cartan elements. In particular, when $L$ is concentrated in non-negative degrees, $\Del_\infty(L)$ is naturally isomorphic to the nerve of the associated Deligne groupoid.  

The $\infty$-groupoid $\Del_\infty(L)$ can be defined as the Kan complex $\Del_\infty(L):=\MC(C^*(\Delta_\bullet;L))$ of (non-degenerate) Maurer-Cartan \emph{cochains} on $\Delta_\bullet$ with coefficients in $L$, where the cochain complex $C^*(\Delta_n;L)$ is equipped with the $L_\infty$ algebra structure induced via homotopy transfer along Dupont's contraction (see \cite{Dupont,GetzlerLie}). Similar constructions appeared in some recent literature \cite{BFMT1,BFMT2,Nicaud}, especially in connection with rational homotopy theory. The equivalence between the above definition of $\Del_\infty(L)$ and the original one from \cite{GetzlerLie} depends on a formal analog of Kuranishi's theorem, explaining how Maurer-Cartan sets behave under homotopy transfer: while this result, which will be proved is Theorem \ref{th:kuranishi}, is essentially due to Getzler, to our knowledge it hasn't been explicitly stated in the literature before. 

The aim of this paper is to prove a generalization of Hinich's descent theorem for the Deligne-Getzler $\infty$-groupoid: more precisely, we prove in Subsection \ref{subs:descent}, Theorems \ref{th:descent} and \ref{th:hldescent}, that the functor $\Del_\infty(-)$ from (complete) $L_\infty$ algebras to Kan complexes commutes (up to homotopy) with certain standard homotopy theoretic constructions, namely, totalization and homotopy limits. We should point out that parts of these results are already contained in Hinich's original paper (cf. \cite[Therem 4.1]{Hinichdescent}), under more restrictive assumptions: on the other hand, our method of proof is rather different and, we believe, more conceptually clear. The main point in our proof is the existence, shown in Theorem \ref{th:Delignevsmappingspaces}, of a natural weak equivalence $\Del_\infty(C^*(X;L))\xrightarrow{\sim}\underline{\mathbf{SSet}}(X,\Del_\infty(L))$ between the Deligne-Getzler $\infty$-groupoid of $C^*(X;L)$ (with the $L_\infty$ algebra structure induced via homotopy transfer along Dupont's contraction) and the simplicial mapping space $\underline{\mathbf{SSet}}(X,\Del_\infty(L))$: this allows to prove the compatibility between $\Del_\infty(-)$ and $\Tot(-)$ by working inductively on the usual tower of partial totalizations. Theorem \ref{th:Delignevsmappingspaces} is a slight generalization of results from \cite{BrSz} and \cite{BerglundLie}, where some additional restrictions on $X$ or $L$ are imposed: once again, our method of proof is rather different from the one in the above references, and depends on some standard results on Reedy model categories.

\begin{ack} Almost all the results proved here are already contained in the author's PhD Thesis \cite{tesi}: we are grateful to our PhD advisor, Marco Manetti, for proposing to us these problems and his unvaluable support. We are also grateful to Domenico Fiorenza, Florian Sch\"atz, Francesco Meazzini, Ping Xu, Mathieu Sti\'enon and Daniel Robert-Nicoud for many useful discussions. This paper was written while the author worked for a semester at Penn State University: we are grateful to the institution for the excellent working conditions.
\end{ack}

\section{Formal Kuranishi theorem} 
\subsection{Algebraic preliminaries}\label{subs: preliminaries} In this section we review some basic algebraic material, mainly with the aim to fix some terminology and notations. Throughout the paper, we work over a field $\K$ of characteristic zero. A graded space $V$ is graded over the integers, $V=\oplus_{k\in \mathbb{Z}}V^k$: given a homogeneous element $v\in V$, we denote by $|v|$ its degree. Given a graded space $V$, we denote by $V[1]$ its \emph{desuspension}, which is the graded space defined by $V[1]^k=V^{k+1}$. Differentials raise the degree by one. Given a category $\mathbf{C}$ and objects $X,Y$ in it, we shall denote by $\mathbf{C}(X,Y)$ the set of morphism between $X$ and $Y$ in $\mathbf{C}$.

Given a graded space $V$, we denote by $V^{\odot n}$ the $n$-th symmetric power of $V$, that is, the quotient of $V^{\otimes n}$ by the subspace spanned by the elements $v_1\otimes\cdots \otimes v_n - \varepsilon(\sigma) v_{\sigma(1)}\otimes \cdots\otimes v_{\sigma(n)}$, where $\sigma\in S_n$ is a permutation and $\varepsilon(\sigma)$ is the usual Koszul sign. We denote by $v_1\odot\cdots\odot v_n$ the image of $v_1\otimes\cdots\otimes v_n$ under the projection $V^{\otimes n}\to V^{\odot n}$. Finally, we denote by $\overline{S}(V)=\oplus_{n\geq1}V^{\odot n}$ the reduced symmetric coalgebra over $V$, with the usual unshuffle coproduct $\Delta(v_1\odot\cdots\odot v_n) = \sum_{i=1}^{n-1} \sum_{\sigma\in S(i,n-i)} \varepsilon(\sigma)(v_{\sigma(1)}\odot\cdots\odot v_{\sigma(i)})\otimes(v_{\sigma(i+1)}\odot\cdots\odot v_{\sigma(n)})$, where $S(i,n-i)\subset S_n$ is the set of $(i,n-i)$-unshuffles, that is, permutations $\sigma\in S_n$ such that $\sigma(j)<\sigma(j+1)$ for $j\neq i$. Given graded spaces $W, V$ and a map $F:\bS(W)\to\bS(V)$, we shall denote by $f_i:W^{\odot i}\to V$ the composition $W^{\odot i}\hookrightarrow\bS(W)\xrightarrow{F} \bS(V)\twoheadrightarrow V$, where the rightmost map is the canonical projection.

The coalgebra $\overline{S}(V)$ is coassociative, cocommutative and locally conilpotent, which means that $\overline{S}(V)=\bigcup_{n\geq1}\operatorname{Ker}(\Delta^{n})$, where $\Delta^{n}:\overline{S}(V)\to\overline{S}(V)^{\otimes n+1}$ is the iterated coproduct.  It is well known that $\overline{S}(V)$ is the cofree object over $V$ in the category of coassociative, cocommutative, locally conilpotent graded coalgebras, thus every coderivation $Q:\bs(V)\to \bs(V)$ and every morphism of graded coalgebras $F:\bs(W)\to\bs(V)$ are uniquely determined by their corestrictions $q:=\sum_{i\geq1}q_i:\bs(V)\to V$ and $f:=\sum_{i\geq1}f_i:\bs(W)\to V$: we shall call the maps $q_i$ and $f_i$ the Taylor coefficients of $Q$ and $F$ respectively. The coderivation $Q$ is determined by its Taylor coefficients $(q_1,\ldots,q_i,\ldots)$ via the formula 
\begin{equation}\label{codfromtaylor}
 Q(v_1\odot\cdots\odot v_n)=\sum_{i=1}^{n} \sum_{\sigma\in S(i,n-i)}\varepsilon(\sigma) q_i(v_{\sigma(1)}\odot\cdots\odot v_{\sigma(i)})\odot v_{\sigma(i+1)}\odot\cdots\odot v_{\sigma(n)},  \end{equation}
while the morphism $F$ is determined by its Taylor coefficients $(f_1,\ldots,f_i,\ldots)$ via 
\begin{equation}\label{morfromtaylor} 
 F(w_1\odot\cdots\odot w_n)=\sum_{k=1}^{n}\frac{1}{k!}\sum_{\stackrel{i_1,\ldots,i_k\ge1}{i_1+\cdots+i_k=n}} \sum_{\sigma\in S(i_1,\ldots,i_k)}\varepsilon(\sigma) f_{i_1}(w_{\sigma(1)}\odot\cdots)\odot\cdots\odot f_{i_k}( \cdots\odot w_{\sigma(n)}).  \end{equation}
In particular, corestriction induces an isomorphism of graded spaces $\Coder(\bs(V))\cong\Hom(\bs(V),V)$, where we denote by $\Hom(-,-)$ the internal Hom in the category of graded spaces, and an isomorphism of sets $\mathbf{GCC}(\bs(W),\bs(V))\cong\mathbf{G}(\bs(W),V)$, where $\mathbf{G}$ and $\mathbf{GCC}$ are the categories of graded spaces and graded cocommutative coassociative locally conilpotent coalgebras respectively.

\begin{definition} An \emph{$L_\infty$ algebra} $(V,Q)$ is a graded space $V$ together with a dg coalgebra structure $Q$ on $\overline{S}(V[1])$, that is, a family of degree one maps $q_i:V[1]^{\odot i}\to V[1]$, $i\geq1$, such that the corresponding coderivation $Q\in\Coder(\bs(V[1]))$ squares to zero. 
	
An \emph{$L_\infty$ morphism} $F:(W,R)\to(V,Q)$ between $L_\infty$ algebras is a morphism of dg coalgebras $F:\bs(W[1])\to\bs(V[1])$, that is, a family of degree zero maps $f_i:W[1]^{\odot i}\to V[1]$, $i\geq1$, such that the corresponding $F:\bs(W[1])\to\bs(V[1])$ commutes with $R$ and $Q$. 

Finally, a \emph{strict $L_\infty$ morphism} $f:(W, R)\to(V,Q)$ is a map of graded spaces $f:W\to V$ such that the morphism $F:\bs(W[1])\to\bs(V[1])$ defined by $f_1=f$, $f_i=0$ for $i\geq2$, is an $L_\infty$ morphism from $(W,R)$ to $(V,Q)$. Equivalently, $f$ is a strict $L_\infty$ morphism if the relation $f(r_i(w_1\odot\cdots\odot w_i))=q_i(f(w_1)\odot\cdots\odot f(w_i))$ holds for all $i\geq1$.

We shall denote by $\sL_\infty$ (resp.: $\mathbf{L}_\infty$) the category of $L_\infty$ algebras and (resp.: strict) $L_\infty$ morphism between them.\end{definition}

\begin{remark} To simplify the notations, and since we won't need to do many explicit computations, we shall denote by the same symbol an element $v\in V$ and its image in $V[1]$ under the shift map. When we do so, it should be clear from the context whether we are considering $v$ as an element in $V$ or $V[1]$. The distinction becomes relevant when applying Koszul rule for switching signs.
\end{remark}

\begin{remark}\label{tangent complex} It follows from the definitions that if $Q=(q_1,\ldots,q_i,\ldots)$ induces an $L_\infty$ algebra algebra structure on $V$, its linear part $q_1:V[1]\to V[1]$ squares to zero: we call the complex $(V[1],q_1)$ the \emph{tangent complex} of $(V,Q)$. Always from the definitions, given an $L_\infty$ morphism $F=(f_1,\ldots,f_i,\ldots):(W,R)\to(V,Q)$, its linear part $f_1$ is a morphism between the respective tangent complexes $f_1:(W[1],r_1)\to(V[1],q_1)$. An $L_\infty$ morphism $F$ is a \emph{weak equivalence} if its linear part $f_1$ is a quasi-isomorphism between the tangent complexes. 
\end{remark}

\begin{remark} There is another equivalent definition of $L_\infty$ algebras and $L_\infty$ morphisms found frequently in the  literature, see for instance the references \cite{GetzlerLie} and \cite{LadaStasheff}. The two are related by the d\'ecalage isomorphism $\mbox{d\'ec}:\Hom^k(W[1]^{\odot i}, V[1])\cong \Hom^{k+1-i}(W^{\wedge k}, V)$, where $W^{\wedge k}$ are the exterior powers of $W$: cf. \cite{FMcone} for our conventions on d\'ecalage. 
	
In particular, every dg Lie algebra $(L,d,[-,-])$ can be regarded as an $L_\infty$ algebras via the coderivation $q_1(l)= -dl$, $q_2(l_1,l_2)=(-1)^{|l_1|}[l_1,l_2]$, $q_i=0$ for $i\geq3$, and $f:L\to M$ is a strict $L_\infty$ morphism if and only if it is a morphism of dg Lie algebras. In other words, there is a canonical full embedding $\mathbf{DGLA}\subset \mathbf{L}_\infty$, where we denote by $\mathbf{DGLA}$ the category of dg Lie algebras.
\end{remark}

Throughout the paper, we shall work in a complete setting.

\newcommand{\dlim}{\underrightarrow{\operatorname{lim}}}
\newcommand{\ilim}{\underleftarrow{\operatorname{lim}}}

\begin{definition} A \emph{complete graded space} is a graded space $V$ equipped with a descending filtration $F^\bullet V$, 
	\[ V=F^1 V\supset\cdots\supset F^pV\supset\cdots  \]
such that $V$ is complete in the induced topology, that is, the natural $V\to\ilim\, V/F^pV$ is an isomorphism of graded spaces. Given complete graded spaces $(W,F^\bullet W)$ and $(V,F^\bullet V)$, a map of graded spaces $f:W\to V$ is \emph{continuous} if $f(F^pW)\subset F^pV$ for all $p\geq1$. We shall denote by $\widehat{\mathbf{G}}$ the category of complete graded spaces and continuous morphisms between them.

A \emph{complete dg space} $(V,F^\bullet V,d)$ is a complete graded space $(V,F^\bullet V)$ equipped with a continuous differential $d$. We shall denote by $\widehat{\mathbf{DG}}$ the category of complete dg spaces and continuous morphisms between them.

A \emph{complete $L_\infty$ algebra} is a complete graded space $(V,F^\bullet V)$ together with an $L_\infty$ algebra structure $Q$ on $V$ such that the Taylor coefficients $q_i$ are continuous for the induced topology, that is $q_i(F^{p_1}V[1]\odot\cdots \odot F^{p_i}V[1])\subset  F^{p_1+\cdots+p_i}V[1]$, for all $i,p_1,\ldots,p_i\geq 1$. 

Similarly, a \emph{continuous $L_\infty$ morphism} $F:(W,F^\bullet W,R)\to (V,F^\bullet V, Q)$ between complete $L_\infty$ algebras is an $L_\infty$ morphism $F:(W,R)\to(V,Q)$ such that its Taylor coefficient are continuous, that is, $f_i(F^{p_1}W[1]\odot\cdots \odot F^{p_i}W[1])\subset F^{p_1+\cdots+p_i}V[1]$, for all $i,p_1,\ldots,p_i\geq 1$. 

We shall denote by $\widehat{\sL_\infty}$ (resp.: $\widehat{\mathbf{L}_\infty}$) the category of $L_\infty$ algebras and continuous (resp.: strict) $L_\infty$ morphisms between them.
\end{definition}

\newcommand{\sR}{\mathcal{R}}

We close this section by recalling the definition of the Maurer-Cartan functor \[ \MC(-):\widehat{\sL_\infty}\to\mathbf{Set}, \]
where we denote by $\mathbf{Set}$ the category of sets. First, given a complete $L_\infty$ algebra $(V,F^\bullet V,Q)$, its \emph{curvature} is the map of sets $\sR_V:V^1\to V^2$ defined by 
\[ \sR_V(x) = \sum_{i\geq1} \frac{1}{i!}q_i(\overbrace{ x\odot\cdots\odot  x}^i)\qquad\mbox{for all $x\in V^1$}. \]
The above infinite sum (and the following ones) is convergent by completeness of $(V,F^\bullet V,Q)$. 

Given a continuous $L_\infty$ morphism $F:(W,F^\bullet W,R)\to(V,F^\bullet V,Q)$ of complete $L_\infty$ algebras, the associated \emph{push-forward} is the map of sets $F_\ast:W^1\to V^1$ defined by 
\[ F_\ast(x) = \sum_{i\geq1} \frac{1}{i!}f_i(\overbrace{x\odot\cdots\odot x}^i)\qquad\mbox{for all $x\in W^1$}. \] 

The curvature and the push-forward are related by the following identity, which can be shown via a direct computation (we omit the details, as this is well known)
\begin{equation}\label{Bianchi} \sR_V( F_\ast (x)) = \sum_{i\ge0}\frac{1}{i!}\, f_{i+1}(\sR_W(x)\odot\overbrace{x\odot\cdots\odot x}^i),\qquad\mbox{for all $x\in W^1$}. \end{equation}
\begin{definition} Given a complete $L_\infty$ algebra $(V,F^\bullet V,Q)$, its \emph{Maurer-Cartan set} is the set
	\[ \MC(V):=\{ x\in V^1\,\, \operatorname{s.t.}\,\, \sR_V(x)=0\}.\]
Given a continuous $L_\infty$ morphism $F:(W,F^\bullet W,R)\to(V,F^\bullet V,Q)$ of complete $L_\infty$ algebras, the associated morphism of Maurer-Cartan sets is the restriction of the push-forward $\MC(F):= F_{\ast|\MC(W)}:\MC(W)\to\MC(V)$. The fact that this takes values in $\MC(V)$ is a consequence of the above identity \eqref{Bianchi}.

\end{definition}

\subsection{Homotopy transfer and formal Kuranishi theorem} In this section we review the homotopy transfer theorem for (complete) $L_\infty$ algebras, together with a formal analog of Kuranishi's theorem, saying how Maurer-Cartan sets behave under homotopy transfer. 

\begin{definition} A \emph{complete contraction} \[\xymatrix{ W\ar@<2pt>[r]^f& V\ar@<2pt>[l]^g\ar@(rd,ru)[]_K }\] is the data of a complete dg space $(V,F^\bullet V,d_V)$ and a dg space $(W,d_W)$, together with dg morphisms $f:(W,d_W)\to (V,d_V)$, $g:(V,d_V)\to(W,d_W)$ and a contracting (degree minus one) homotopy $K:V\to V$, such that 
	\begin{itemize} \item $g$ is a left inverse to $f$, that is, $gf=\id_W$;
		\item $K$ is a homotopy between $fg$ and $\id_V$, that is, $Kd_V+d_VK=fg-\id_V$;
		\item $K$ satisfies the side conditions $Kf = K^2 = gK= 0$;
		\item $K$ and $fg$ are continuous with respect to the filtration $F^\bullet V$ on $V$.
	\end{itemize}
We shall always equip $W$ with the induced filtration $F^pW = f^{-1}(F^pV)$: the last condition ensures that $(W,d_W)$ is a complete dg space with respect to this filtration, and $f,g$ are continuous morphism.

Given two complete contractions $\xymatrix{ W\ar@<2pt>[r]^f& V\ar@<2pt>[l]^g\ar@(rd,ru)[]_K }$ and $\xymatrix{ {W'}\ar@<2pt>[r]^{f'}& V'\ar@<2pt>[l]^{g'}\ar@(rd,ru)[]_{K'}}
$ as above, a morphism between them is a continuous morphism of dg spaces $\phi:(V,d_V)\to (V',d_V')$ commuting with the homotopies $K$ and $K'$, that is, $K'\phi=\phi K$.

With these definitions, complete contractions form a category, which we denote by $\widehat{\mathbf{Ctr}}$. 
\end{definition}

\begin{remark}\label{rem:pr} There is a pair of functors $\operatorname{pr}_i:\widehat{\mathbf{Ctr}}\to\widehat{\mathbf{DG}}$, $i=1,2$. The functor $\operatorname{pr}_1$ (resp.: $\operatorname{pr}_2$) sends a complete contraction $\xymatrix{ W\ar@<2pt>[r]^f& V\ar@<2pt>[l]^g\ar@(rd,ru)[]_K }$ to the complete dg space $(V,F^\bullet V,d_V)$ (resp.: $(W,F^\bullet W,d_W)$) and a morphism $\phi: \left(\xymatrix{ W\ar@<2pt>[r]^f& V\ar@<2pt>[l]^g\ar@(rd,ru)[]_K }\right)\,\to\,\left(\xymatrix{ {W'}\ar@<2pt>[r]^{f'}& V'\ar@<2pt>[l]^{g'}\ar@(rd,ru)[]_{K'}}\right)$ to the continuous dg morphism $\phi:(V,d_V)\to (V',d_{V'})$ (resp.: $g'\phi f:(W,d_W)\to (W',d_{W'})$). The fact that $\operatorname{pr}_2$ is a functor is not obvious, but follows easily from the definitions. It is easy to check that the category $\widehat{\mathbf{Ctr}}$ is complete, and that the functors $\operatorname{pr}_1, \operatorname{pr}_2$ commute with small limits.
\end{remark}

The \emph{homotopy transfer theorem} asserts that (complete) $L_\infty$ algebra structures can be transferred along (complete) contractions.

\begin{theorem}\label{th:transfer} Given a complete contraction $\xymatrix{ W[1]\ar@<2pt>[r]^{f_1}& V[1]\ar@<2pt>[l]^{g_1}\ar@(rd,ru)[]_K }$ and a complete $L_\infty$ algebra structure $Q$ on $(V,F^\bullet V)$ with linear part $q_1=d_{V[1]}$, there is an induced complete $L_\infty$ algebra structure $R$ on $(W,F^\bullet W)$ with linear part $r_1=d_{W[1]}$, together with continuous $L_\infty$ morphisms $F:(W,R)\to(V,Q)$, $G:(V,Q)\to (W,R)$ with linear parts $f_1$, $g_1$ respectively. Denoting by $F^k_i$ the composition $W[1]^{\odot i}\hookrightarrow\bs(W[1])\xrightarrow{F}\bs(V[1])\twoheadrightarrow V[1]^{\odot k}$, $F$ and $R$ are detemined recursively by (notice that by formula \eqref{morfromtaylor} $F^k_i$ only depends on $f_1,\ldots, f_{i-k+1}$)
\[ f_i = \sum_{k=2}^i Kq_k F^k_i \qquad \mbox{for $i\geq2$}, \]	
\[ r_i = \sum_{k=2}^i g_1q_k F^k_i \qquad \mbox{for $i\geq2$}. \]	
It is possible to establish recursive formulas for $G$ as well, but these are a bit more complicated. We denote by $K_i:V[1]^{\odot i}\to V[1]^{\odot i}$ the degree minus one map defined by 
\[ K_i(v_1\odot\cdots\odot v_i) = \frac{1}{i!} \sum_{\sigma\in S_i}\sum_{j=1}^{i}\pm f_1g_1(v_{\sigma(1)})\odot\cdots\odot f_1g_1(v_{\sigma(j-1)})\odot K(v_{\sigma(j)})\odot v_{\sigma(j+1)}\odot\cdots\odot v_{\sigma(i)}, \]
where $\pm$ is the appropriate Koszul sign (taking into account that $|K|=-1$).  Denoting by $Q^k_i$ the composition $V[1]^{\odot i}\hookrightarrow\bs(V[1])\xrightarrow{Q}\bs(V[1])\twoheadrightarrow V[1]^{\odot k}$, the $L_\infty$ morphism $G$ is determined recursively by
\[ g_i = \sum_{k=1}^{i-1} g_k Q^k_i K_i\qquad\mbox{for $i\geq2$}.  \]
\end{theorem}
\begin{proof} In the non-complete setting, the first part of the theorem is completely standard: for a proof we refer to the arXiv version of \cite{FMcone}. The second part of the theorem (the one concerning the $L_\infty$ morphism $G$) is less standard: for a proof we refer to \cite{BerglundHPT}.
	
The fact that $R$, $F$ and $G$ are continuous follows by straightforward inductions, using the recursive formulas.	
\end{proof}
We shall need the following lemmas, where the functors $\operatorname{pr}_i:\widehat{\mathbf{Ctr}}\to\widehat{\mathbf{DG}}$, $i=1,2$, were introduced in Remark \ref{rem:pr}. The proofs proceed by easy inductions, using the recursive formulas from Theorem \ref{th:transfer}: details may be found in \cite[Lemma 2.2.3, Lemma 2.2.7]{tesi}.
\begin{lemma}\label{lem:transfer} Given a morphism of complete contractions \[\phi: \left(\xymatrix{ W[1]\ar@<2pt>[r]^f& V[1]\ar@<2pt>[l]^g\ar@(rd,ru)[]_K }\right)\,\to\,\left(\xymatrix{ {W'[1]}\ar@<2pt>[r]^{f'}& V'[1]\ar@<2pt>[l]^{g'}\ar@(rd,ru)[]_{K'}}\right),\] 
together with complete $L_\infty$ algebra structures $Q$ and $Q'$ on $V,V'$ with linear parts $q_1=d_{V[1]}$, $q'_1=d_{V'[1]}$ respectively, if $\operatorname{pr}_1(\phi)$ is a strict $L_\infty$ morphism $\operatorname{pr}_1(\phi):(V,Q)\to(V',Q')$, then also $\operatorname{pr}_2(\phi):(W,R)\to (W',R')$ is a strict $L_\infty$ morphism between the transferred $L_\infty$ structures, and furthermore the following diagrams are commutative, where the $L_\infty$ morphisms $F,F',G,G'$ are defined as in the previous theorem,
\[\xymatrix{(V,Q)\ar[r]^{\operatorname{pr_1}(\phi)} & (V',Q')\\ (W,R)\ar[r]^{\operatorname{pr}_2(\phi)}\ar[u]_F & (W',R')\ar[u]_{F'}}  \qquad \xymatrix{(V,Q)\ar[r]^{\operatorname{pr_1}(\phi)}\ar[d]_G & (V',Q')\ar[d]_{G'}\\ (W,R)\ar[r]^{\operatorname{pr}_2(\phi)} & (W',R')} \]
\end{lemma}

\begin{lemma}\label{lem:transferalongstrictmorphisms} Let $g_1:(V,Q)\rh (W,R)$ be a strict morphism of complete $L_\infty$ algebras, fitting into a contraction $\xymatrix{ W[1]\ar@<2pt>[r]^{f_1}& V[1]\ar@<2pt>[l]^{g_1}\ar@(rd,ru)[]_K }$ from $(V,q_1)$ to $(W,r_1)$. The transferred $L_\infty[1]$ algebra structure on $W$ is again $(W,R)$, moreover the $L_\infty$ morphism $G:(V,Q)\rh (W,R)$ from Theorem \ref{th:transfer} is $G=g_1$.\end{lemma}

\begin{remark}\label{rem:transfer as a functor} By the above Lemma \ref{lem:transfer}, homotopy transfer can be regarded as a functor \[ \widehat{\mathbf{L}_\infty}\times_{\widehat{\mathbf{DG}}}\widehat{\mathbf{Ctr}}\to\widehat{\mathbf{L}_\infty}\]
where the fiber product $\widehat{\mathbf{L}_\infty}\times_{\widehat{\mathbf{DG}}}\widehat{\mathbf{Ctr}}$ is taken along the functor $\operatorname{pr}_1$ from Remark \ref{rem:pr} and the tangent complex functor from Remark \ref{tangent complex}. It is easy to check that when we regard homotopy transfer as a functor in the above sense,  it commutes with small limits.
\end{remark}

The main result of this section is the following formal analog of Kuranishi's theorem, which is essentially due to Getzler \cite{GetzlerLie, Getzlerpert}: in particular, our proof follows closely the proofs of Lemma 4.6 and Lemma 5.3 in loc. cit..

\begin{theorem}\label{th:kuranishi} In the same hypotheses as in Theorem \ref{th:transfer}, the correspondence
\[\rho:\MC(V)\rh\MC(W)\times K(V^1):x\rh(\MC(G)(x),K(x))\]
is bijective. The inverse $\rho^{-1}$ admits the following recursive construction: given $y\in\MC(W)$ and $K(v)\in K(V^1)$, we define a succession of elements $x_n\in V^1$, $n\geq 0$, by $x_{0}=0$ and 
\begin{equation}\label{eq:getzlerrecursion} x_{n+1}=f_1(y)-q_1K(v)+\sum_{i\geq2}\frac{1}{i!}\left(Kq_i-f_1g_i\right)\left(x_n^{\odot i}\right).    \end{equation}
This succession converges (with respect to the complete topology induced by the filtration on $V$) to a well defined $x\in V^1$, and we have  $\rho^{-1}(y,K(v))=x$. Finally, $\rho^{-1}(-,0)$ coincides with $\MC(F):\MC(W)\rh\MC(V)$, which induces a bijective correspondence between the sets $\MC(W)$ and $\op{Ker}\,K\bigcap\MC(V)$, whose inverse is the restriction of $g_1$. \end{theorem}

\begin{proof} We proceed as in \cite{GetzlerLie}, and use the notations from the previous subsection. If $x\in\MC(V)$, then
\begin{equation}\label{eq:pointfixed} x=f_1g_1(x)-q_1K(x)-Kq_1(x)=f_1G_\ast(x)-q_1K(x)+\sum_{i\geq 2}\frac{1}{i!}\left(Kq_i-f_1g_i\right)\left(x^{\odot i}\right). \end{equation}

Equation \eqref{eq:pointfixed} implies injectivity of $\rho$ as follows: if $z\in\MC(V)$ is such that $G_*(x)=G_*(z)$, $K(x)=K(z)$, then subtracting the respective equations \eqref{eq:pointfixed} for $x$ and $z$ we obtain
\[ x-z = \sum_{i\geq2}\frac{1}{i!}\sum_{j=0}^{i-1}\left(Kq_i-f_1g_i\right)\left(x^{\odot j}\odot(x-z)\odot z^{\odot i-j-1}\right).\]
The above shows $x-z\in F^pV\solose x-z\in F^{p+1}V$, thus inductively $x-z\in\bigcap_{p\geq1}F^pV=0$.

Now we consider $y\in\MC(W)$, $K(v)\in K(V^1)$, and the sequence $x_n\in V^0$ defined by the recursion \eqref{eq:getzlerrecursion}: we show that this sequence is convergent. We suppose inductively, starting with $x_1-x_0\in F^1V=V$,  to have proved that $x_{n}-x_{n-1}\in F^{n}V$, and deduce
\[ x_{n+1}-x_n = \sum_{i\geq2}\sum_{j = 0}^{i-1}\frac{1}{i!}(Kq_i-f_1g_i)\left( x_n^{\odot j}\odot(x_n-x_{n-1})\odot x_{n-1}^{\odot i-j-1}\right)\in F^{n+1}V.\]
By completeness, the infinite sum $\sum_{n\geq0}(x_{n+1}-x_n)$ converges to a well defined $x\in V^1$, and by construction this satisfies
\[ x=f_1(y)-q_1K(v)+\sum_{i\geq 2}\frac{1}{i!}\left(Kq_i-f_1g_i\right)\left(x^{\odot i}\right). \]
Applying $K$, since $Kf_1=K^2=g_1K=0$,
\[ K(x) = -Kq_1K(v)=(q_1K+\id_V-f_1g_1)K(v)=K(v).\]
Applying $g_1$, since moreover $g_1f_1=\id_W$, $g_1q_1=r_1g_1$,
\[g_1(x)=y-\sum_{i\geq2}\frac{1}{i!}g_i(x^{\odot i})\qquad\Longrightarrow\qquad G_*(x)=y.\]
Applying $q_1$, we get
\[q_1(x)=q_1f_1(y)+\sum_{i\geq2}\frac{1}{i!}(f_1g_1-\id_V-Kq_1)q_i(x^{\odot i})-\sum_{i\geq2}\frac{1}{i!}q_1f_1g_i(x^{\odot i}).\]
We notice that
\[q_1f_1(y)-\sum_{i\geq2}\frac{1}{i!}q_1f_1g_i(x^{\odot i})=q_1f_1(y)-q_1f_1G_*(x)+q_1f_1g_1(x)=q_1f_1g_1(x)=f_1g_1q_1(x),\]
and substituting into the previous equation, we find
\[ \mathcal{R}_V(x)=f_1g_1\mathcal{R}_V(x)-\sum_{i\geq2}\frac{1}{i!}Kq_1q_i(x^{\odot i}).\]
We denote by $H$ the composition $H:=FG:(V,Q)\to(V,Q)$, whose linear part is  $h_1 =f_1g_1$. By Formula \eqref{Bianchi}, and since $y\in\MC(W)$, we see that
\[ 0 = \mathcal{R}_VF_*(y)=\mathcal{R}_VH_*(x)=h_1\mathcal{R}_V(x)+\sum_{k\geq2}\frac{1}{(k-1)!}h_k(\mathcal{R}_V(x)\odot x^{\odot k-1}).\]
On the other hand, since $Q^2=0\solose \sum_{j=1}^{i}q_jQ^j_i = 0$, and using Formula \eqref{codfromtaylor},
\begin{multline}\nonumber -\sum_{i\geq2}\frac{1}{i!}Kq_1q_i(x^{\odot i})=\sum_{i\geq2}\frac{1}{i!}\sum_{j=1}^{i-1}\frac{i!}{j!(i-j)!}Kq_{i-j+1}(q_j(x^{\odot j})\odot x^{\odot i-j})=\\=\sum_{k\geq2}\frac{1}{(k-1)!} Kq_k(\mathcal{R}_V(x)\odot x^{\odot k-1}).\end{multline}
Finally, putting everything together
\[\mathcal{R}_V(x)=\sum_{k\geq2}\frac{1}{(k-1)!}(Kq_k-h_k)(\mathcal{R}_V(x)\odot x^{\odot k-1}),\]
which shows $\mathcal{R}_V(x)\in F^pV\solose\mathcal{R}_V(x)\in F^{p+1}V$, thus inductively $\mathcal{R}_V(x)=0$. We have constructed $x\in\MC(V)$ such that $\MC(G)(x)=y$ and $K(x)=K(v)$, thus $\rho$ is surjective.

Since $f_{i}=\sum_{j=2}^iKq_jF^j_i$ for $i\geq2$, we see that $g_1f_i=Kf_i=0$ for $i\geq2$, hence the identities
\[ g_1F_*=\id_{W^1},\qquad KF_*=0.\]
Given $y\in\MC(W)$ we have $\MC(G)(\MC(F)(y))=y$ and by the above also $K(\MC(F)(y))=0$, thus $\MC(F)(y)=\rho^{-1}(y,0)$. Together with the first part of the theorem, this makes it clear that $\MC(F)=\rho^{-1}(-,0):\MC(W)\rh\op{Ker}\,K\bigcap\MC(V)$ is a bijective correspondence, and we have already observed that $g_1\MC(F)=\id_{\MC(W)}$.\end{proof}

\begin{remark} Given $F:(W,R)\to(V,Q)$ as in Theorem \ref{th:transfer}, we observe that the previous proof continues to work for any $L_\infty$ morphism $G:(V,Q)\to(W,R)$ left inverse to $F$ (not necessarily the particular one in the claim of Theorem \ref{th:transfer}).
\end{remark}

\section{Deligne-Getzler $\infty$-groupoids}

\subsection{Deligne-Getzler $\infty$-groupoids} Given a category $\mathbf{C}$, we denote by $\mathbf{C}^{op}$ the opposite category. Given a small category $\mathbf{S}$, we denote by $\mathbf{C}^{\mathbf{S}}$ the category of functors $\mathbf{S}\to\mathbf{C}$. In particular, we denote by $\Delta$ the ordinal category, and by $\mathbf{C}^{\Delta}$ (resp.: $\mathbf{C}^{\Delta^{op}}$) the category of cosimplicial (resp.: simplicial) objects in $\mathbf{C}$. We denote the category of simplicial sets by $\mathbf{SSet}$. 

Given a simplicial set $X$ and a dg space $L$, we denote by $\Omega^*(X;L)$ the complex of polynomial differential forms on $X$ with coefficients in $L$, and by $C^*(X;L)$ the complex of non-degenerate simplicial cochains on $X$ with coefficients in $L$: see for instance \cite{RHT}. There is a standard contraction from $\Omega^*(X;L)$ to $C^*(X;L)$, given by integrating forms over simplices in one direction and by the inclusion of Whitney's elementary forms in the other direction, together with a contracting homotopy defined by Dupont \cite{Dupont}. We won't need explicit formulas for this contraction, for which we refer to \cite{Dupont} or \cite{GetzlerLie}, just the fact that it is natural with respect to pull-back by morphism of simplicial sets and push-forward by morphism of dg spaces: more precisely, it is defined a functor $\operatorname{Dup}:\mathbf{SSet}^{op}\times\mathbf{DG}\to\mathbf{Ctr}$. In the complete setting, given a complete dg space $L$, the complex $C^*(X:L)$ is naturally complete with respect to the filtration $F^p C^*(X;L)=C^*(X;F^p L)$. On the other hand, in general the complex $\Omega(X;L)$ doesn't inherit a complete structure, and we have to replace it by its completion $\widehat{\Omega}^*(X;L):=\ilim\,\Omega^*(X;L/F^pL)$: this is a typical source of tedious, and ultimately unimportant, technicalities. Taking the limit of the contractions 
$\xymatrix{ C^*(X;L/F^pL)\ar@<2pt>[r]& \Omega^*(X;L/F^p L)\ar@<2pt>[l]}$ yields a complete contraction $\xymatrix{ C^*(X;L)\ar@<2pt>[r]& \widehat{\Omega}^*(X;L)\ar@<2pt>[l] }$, where we omitted the homotopies for notational simplicity, and a functor $\operatorname{Dup}:\mathbf{SSet}^{op}\times\widehat{\mathbf{DG}}\to\widehat{\mathbf{Ctr}}$. If $L$ is a complete $L_\infty$ algebra, so is $\widehat{\Omega}^*(X;L)$, by extension of scalars, and there is an induced complete $L_\infty$ algebra structure on $C^*(X;L)$ via homotopy transfer along Dupont's contraction. Furthermore, by Lemma \ref{lem:transfer}, the pull-back by morphisms of simplicial sets and the push-furward by \emph{strict} $L_\infty$ morphisms are strict $L_\infty$ morphisms, hence it is defined the functor of non-degenerate simplicial cochains $C^*(-;-): \mathbf{SSet}^{op}\times\widehat{\mathbf{L}_\infty}\to\widehat{\mathbf{L}_\infty}$. 

\begin{remark}\label{rem:limits} The functor $C^*(-;-): \mathbf{SSet}^{op}\times\widehat{\mathbf{L}_\infty}\to\widehat{\mathbf{L}_\infty}$ commutes with small limits: this can be checked using Remark \ref{rem:transfer as a functor} and the fact that so does $C^*(-;-): \mathbf{SSet}^{op}\times\widehat{\mathbf{DG}}\to\widehat{\mathbf{DG}}$.
\end{remark}
\begin{remark} Since the natural $\Omega^*(X;L)\to\widehat{\Omega}^*(X;L)$ commutes with Dupont's contracting homotopies, one can use Lemma \ref{lem:transfer} to deduce that the $L_\infty$ algebra structure on $C^*(X;L)$ induced via homotopy transfer from $\Omega^*(X;L)$ (which is not a complete space but is still an $L_\infty$ algebra) is automatically continuous, and in fact coincides with the above one induced via homotopy transfer from $\widehat{\Omega}^*(X;L)$.
\end{remark}

We shall denote by 
\[ \Delta_{\bullet}:\quad \xymatrix{ \Delta_0
	\ar@<2pt>[r]\ar@<-2pt>[r] & \Delta_1    \ar@<4pt>[r] \ar[r] \ar@<-4pt>[r] & \Delta_2 \ar@<6pt>[r] \ar@<2pt>[r] \ar@<-2pt>[r] \ar@<-6pt>[r]&\cdots}  \]
the standard cosimplicial simplex in $\mathbf{SSet}^\Delta$ (where we are omitting the codegeneracies for notational simplicity). By the above, there is a well defined functor $C^*(\Delta_\bullet;-): \widehat{\mathbf{L}_\infty}\xrightarrow{}\widehat{\mathbf{L}_\infty}^{\Delta^{op}}$, sending a (complete) $L_\infty$ algebra $L$ to the simplicial (complete) $L_\infty$ algebra of cochains on $\Delta_\bullet$ with coefficients in $L$.

\begin{definition}\label{def:higher deligne} Given a complete $L_\infty$ algebra $L$, the \emph{Deligne-Getzler $\infty$-groupoid} of $L$ is the simplicial set $\Del_\infty(L)_\bullet:=\MC(C^*(\Delta_\bullet;L))$ of Maurer-Cartan cochains on $\Delta_\bullet$ with coefficients in $L$. In other words, the functor $\Del_\infty(-):\widehat{\mathbf{L}_\infty}\to\mathbf{SSet}$ is the composition
\[ \Del_\infty(-): \widehat{\mathbf{L}_\infty}\xrightarrow{C^*(\Delta_\bullet;-)}\widehat{\mathbf{L}_\infty}^{\Delta^{op}}\xrightarrow{\MC(-)}\mathbf{SSet}.\]
\end{definition}

\begin{remark}\label{ptu} We observe that the simplicial set $\Del_\infty(L)$ only depends on the $L_\infty[1]$ structure $Q$ on $L$, and not on the particular filtration $F^\bullet L$ making $(L,Q)$ into a complete $L_\infty$ algebra.\end{remark}

The simplicial set $\Del_\infty(L)$ was introduced by Getzler \cite{GetzlerLie}, although with a slightly different definition. More precisely, for a nilpotent $L_\infty$ algebra $L$, Getzler defines a simplicial set $\gamma(L)$ by declaring the $n$-simplices to be the Maurer-Cartan forms $\omega\in\MC(\Omega^*(\Delta_n;L))$ in the kernel of Dupont's contracting homotopy. This can be straightforwardly extended to complete $L_\infty$ algebras by replacing $\Omega^*(\Delta_n;L)$ with $\widehat{\Omega}^*(\Delta_n;L)$. The following proposition is an immediate consequence of Theorem \ref{th:kuranishi}.
\begin{proposition} There is a natural isomorphism $\Del_\infty(-)\xrightarrow{\cong}\gamma(-)$ of functors $\widehat{\mathbf{L}_\infty}\to\mathbf{SSet}$.\end{proposition}

\begin{definition} Given a simplicial set $X$, we denote by $\Delta X$ the (small) category of simplices of $X$ (cf. \cite[Example 15.1.14]{HirMC}): its objects are the morphisms $\sigma:\Delta_n\to X$ in $\mathbf{SSet}$ (which are in bijective correspondence with the $n$-simplices $x\in X_n$ of $X$), for some $n\geq0$; the arrows are the morphisms $\Delta_m\to\Delta_n$ over $X$. We recall (cf. \cite[Proposition 15.1.20]{HirMC}) that there is a natural isomorphism $X=\dlim_{\{\sigma:\Delta_n\to X\}\in\Delta X}\Delta_n$.
\end{definition}

We notice the following formal consequence of Definition \ref{def:higher deligne}, which anticipates the more profound relationship between $\Del_\infty(-)$ and simplicial mapping spaces in Theorem \ref{th:Delignevsmappingspaces}.

\begin{lemma}\label{prop:MC(C*(-;-))} There is a natural isomorphism $\MC(C^*(-;-))\xrightarrow{\cong}\mathbf{SSet}(-,\Del_\infty(-))$ of functors $\mathbf{SSet}^{op}\times\widehat{\mathbf{L}_\infty}\to\mathbf{Set}$.
\end{lemma}

\begin{proof} Let $L$ be a complete $L_\infty$ algebra and $X$ a simplicial set. Using Remark \ref{rem:limits}, there is a strict $L_\infty$ isomorphism $C^*(X;L)\to\ilim_{\{\sigma:\Delta_n\to X\}\in(\Delta X)^{op}}C^*(\Delta_n;L)$. Since $\MC(-)$ commutes with small limits,
	\[\MC(C^*(X;L))=\ilim_{(\Delta X)^{op}} \,\MC(C^*(\Delta_n;L))=:\ilim_{(\Delta X)^{op}}\,\Del_\infty(L)_{n} = \]
	\[=\ilim_{(\Delta X)^{op}}\,\mathbf{SSet}(\Delta_n,\Del_\infty(L))=\mathbf{SSet}(\dlim_{\Delta X}\Delta_{n},\Del_\infty(L))=\mathbf{SSet}(X,\Del_\infty(L)).\]
It is easy to check that this identification is  natural in $X$ and $L$.\end{proof}

There are some technical and conceptual advantages in working with cochains rather than with forms. For instance, we have the following useful lemma, which follows from the fact that both the functors $C^*(\Delta_\bullet;-):\widehat{\mathbf{L}_\infty}\to\widehat{\mathbf{L}_\infty}^{\Delta^{op}}$ and $\MC(-):\widehat{\mathbf{L}_\infty}^{\Delta^{op}}\to\mathbf{SSet}$ commute with small limits.
\begin{lemma}\label{lem:Delvslimits} The functor $\Del_\infty(-):\widehat{\mathbf{L}_\infty}\to\mathbf{SSet}$ commutes with small limits.\end{lemma}

To illustrate the conceptual advantages, we show how Lemma 5.3 from \cite{GetzlerLie} becomes more transparent in this setting. Given a cochain $\beta\in C^i(\Delta_n;L)$ and $0\leq i_0<\cdots< i_k\leq n$, we denote by $\beta_{i_0\cdots i_k}\in L^{i-k}$ the evaluation of $\beta$ on the $k$-simplex of $\Delta_n$ spanned by the vertices $i_0,\ldots,i_k$. For $i=0,\ldots,n$, we define a homotopy $h^i:C^*(\Delta_n;L)\to C^{*-1}(\Delta_{n};L)$ by
\begin{equation*}  h^i(\alpha)_{i_0\cdots i_k}=\left\{\begin{array}{ll} 0 &\mbox{if $i\in\{i_0,\cdots,i_k\}$}\\ (-1)^j\alpha_{i_0\cdots i_{j-1}ii_{j}\cdots i_{k}}&\mbox{if $0\leq i_0<\cdots<i_{j-1}<i<i_j<\cdots<i_k\leq n$}\end{array}\right.
\end{equation*}
We denote by $e_i:\Delta_0\to\Delta_n$ the inclusion of the $i$-th vertex of $\Delta_n$ and by $\pi:\Delta_n\to\Delta_0$ the final morphism. It is easy to check that the above operator $h^i$ fits into a complete contraction
\[\xymatrix{L=C^*(\Delta_0;L)\ar@<2pt>[r]^-{\pi^\ast}& C^*(\Delta_n;L) \ar@<2pt>[l]^-{e_i^\ast}} \]
If $\de_i:\Delta_{n-1}\to\Delta_n$ is the inclusion of the $i$-th face of the simplex $\Delta_n$, then $\de_i^\ast$ sends $h^i(C^{1}(\Delta_n;L))$ isomorphically onto $C^{0}(\Delta_{n-1};L)$. Lemma \ref{lem:transferalongstrictmorphisms} and Theorem \ref{th:kuranishi} imply the following result.

\begin{proposition}\label{prop:Delinsoldoni} For all $i=0,\ldots,n$, the correspondence
	\[\rho^i:\Del_\infty(L)_n\to\MC(L)\times C^0(\Delta_{n-1};L):\alpha\to (e_i^\ast(\alpha),\de_i^\ast h^i(\alpha))\]
	is bijective.\end{proposition}

\begin{remark}\label{rem:BCH} We can visualize the previous result as follows: given $n\geq0$ and $0\leq i\leq n$, we assign a Maurer-Cartan element to the $i$-th vertex $e_i$ of $\Delta_n$, and an element in $L^{1-k}$ to every $k$-simplex of $\Delta_n$ containing $e_i$, as in the following figure, where $n=2$, $i=1$, $x\in\MC(L)$, $a,b\in L^0$ and $\eta\in L^{-1}$,
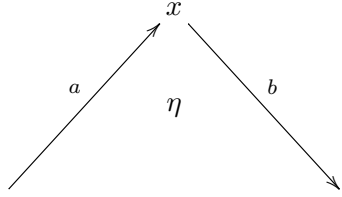
\[ \xymatrix{ & & x\ar[rrdd]^b & &\\ & & \eta & &\\ \ar[uurr]^a & & & &}\]
The previous proposition tells us that for any such assignment there is a unique Maurer-Cartan cochain $\alpha\in\Del_\infty(V)_n$ evaluating to the given elements in the open star around $e_i$. Evaluating this cochain on the face $\de_i\Delta_n$ opposite to $e_i$ defines the higher Baker-Campbell-Hausdorff products $\rho^x_n(-)$ introduced by Getzler in \cite[Definition 5.5]{GetzlerLie},
\[ \xymatrix{ & & x\ar[rrdd]^b & & & & & & & x\ar[rrdd]^b & &\\ & & \eta & & & \ar@{~>}[r]& & & & \eta & &\\ \ar[uurr]^a & & & & & & & \rho^x_1(a)\ar[uurr]^a\ar[rrrr]_{\rho_2^x(a,b,\eta)} & & & &\rho^x_1(b)}\]
When $L$ is a complete dg Lie algebra, the Gauge action and the Baker-Campbell-Hausdorff product on $L$ can be recovered as particular instances of the previous functions $\rho^x_n(-)$, cf. \cite{GetzlerLie} or \cite{tesi} for details.
\end{remark}

As a Corollary, or via a similar reasoning (cf. \cite[Theorem 5.2.10]{tesi}), we deduce the following 
\begin{corollary}\label{th:surjective=Kanfibration} A strict morphism $f:L\to M$ of complete $L_\infty$ algebras induces a Kan fibration $\Del_\infty(f):\Del_\infty(L)\to\Del_\infty(M)$ of simplicial sets if and only if it is surjective in degrees~$\leq0$. In particular, the functor $\Del_\infty(-)$ factors through the full subcategory $\mathbf{Kan}\subset\mathbf{SSet}$ of Kan complexes.\end{corollary}

\subsection{Deligne-Getzler $\infty$-groupoids and models of mapping spaces} The aim of this subsection is to review, and slightly generalize, results from \cite{BrSz,BerglundLie} relating the functor $\Del_\infty(-)$ and simplicial mapping spaces. We start with some preliminary results.

\begin{definition} A central extension $0\to K\to L\to M\to 0$ of complete $L_\infty$ algebras is the datum of a surjective, continuous, strict $L_\infty$ morphism $L\to M$ such that its kernel $K\subset L$ is an abelian $L_\infty$ ideal. Denoting by $q_n:L[1]^{\odot n}\to L[1]$ the Taylor coefficients of the $L_\infty$ structure on $L$, the latter requirement means that for all $n\geq2$ the map $q_n$ vanishes whenever one of its arguments is in $K[1]$. We call $K,L$ and $M$ respectively the fiber, the total space and the base of the $L_\infty$ extension.
\end{definition} 

\begin{proposition}\label{prop:obstructionforMC} Given $0\rh K\rh L\rh M\rh 0$ a central extension of complete $L_\infty$ algebras, there is an obstruction map $o:\MC(M)\rh H^2(K)$ with the property $o(x)=0$ if and only if $x$ lifts to a Maurer-Cartan element $y\in\MC(L)$. If the set of Maurer-Cartan liftings of $x$ is not empty, it has the structure of an affine space over $Z^1(K)$: more precisely, given a Maurer-Cartan lifting $y\in\MC(L)$ of $x$, the set of all Maurer-Cartan liftings of $x$ is in bijective correspondence with $Z^1(K)$ via $Z^1(K)\rh\MC(L):z\rh y+z$. 
\end{proposition}

\begin{proof} We denote by $q_n:L[1]^{\odot n}\rh L[1]$ the Taylor coefficients of the $L_\infty$ structure on $L$. Given $x\in\MC(M)$, let $y\in L^1$ be an arbitrary (not necessarily Maurer-Cartan) lifting of $x$: then $\mathcal{R}_L(y)=\sum_{n\geq1} \frac{1}{n!}q_n(y^{\odot  n})\in Z^2(K)$ (see Subsection \ref{subs: preliminaries} for the notation). In fact, since $x\in\MC(M)$ it is clear that $\mathcal{R}_L(y)\in K$, and since $K\subset L$ is an abelian ideal we see that
	\begin{eqnarray*} q_1(\mathcal{R}_L(y))=\sum_{n\geq2}\frac{1}{n!}q_1q_n(y^{\odot n})&=&-\sum_{n\geq2}\frac{1}{n!}\sum_{i=2}^n\frac{n!}{(i-1)!(n-i+1)!}q_i(q_{n-i+1}(y^{\odot n-i+1})\odot y^{\odot i-1})\\ 
	&=&-\sum_{i\geq2}\frac{1}{(i-1)!}q_i(\mathcal{R}_L(y)\odot y^{\odot i-1})=0
\end{eqnarray*}
	If $\widetilde{y}$ is another lifting of $x$ and $n\geq2$, then \[q_n(y^{\odot n})-q_n(\widetilde{y}^{\odot n})=\sum_{i=0}^{n-1}q_n(y^{\odot i}\odot(y-\widetilde{y})\odot\widetilde{y}^{\odot n-i-1})=0,\] as $y-\widetilde{y}\in K$ and $K\subset L$ is an abelian ideal, thus
	\begin{equation}\label{eq:R(x-x')}\mathcal{R}_L(y)-\mathcal{R}_L(\widetilde{y})=q_1(y-\widetilde{y}).\end{equation}
	This shows that the map $o:\MC(M)\to H^2(K)$ sending $x$ to the cohomology class of $\mathcal{R}_L(y)$, where $y\in L^1$ an arbitrary lifting of $x$, is well defined. If $o(x)=0$ then $\mathcal{R}_L(y)=q_1(z)$ for some $z\in K^1$, and \eqref{eq:R(x-x')} implies that $\mathcal{R}_L(y-z)=0$: thus $x$ admits a Maurer-Cartan lifting, and the converse is obvious. Finally, the last statement also follows immediately from Equation \eqref{eq:R(x-x')}.\end{proof}

\begin{remark}\label{rem:naturalityofobstructions} The obstruction map defined in the previous proposition is natural with respect to strict morphisms between central extensions of complete $L_\infty[1]$ algebras, that is, the datum of continuous strict morphisms between the bases, the fibers and the total spaces making the obvious diagram commutative: this is immediate by construction.\end{remark}

In the following proposition \ref{prop:delignevsexactsequences} we will prove an analog result for $\Del_\infty(-)$. First we recall the following observation, due to Getzler, showing that the functor $\Del_\infty(-)$ can be regarded as a non-abelian analogue of the Dold-Kan functor, cf. the introduction of \cite{GetzlerLie}.
\newcommand{\sMC}{\mathcal{MC}}

\begin{remark}\label{ex:deligneabeliancase} Let $(L,q_1,0,\ldots,0,\ldots)$ be an abelian $L_\infty$ algebra. Since extension of scalars and homotopy transfer send abelian $L_\infty$ structures to abelian $L_\infty$ structures, every $C^*(\Delta_n;L)$ is an abelian $L_\infty$ algebra, and $\Del_\infty(L)$ is the simplicial vector space $\Del_\infty(L)_n=Z^1(C^*(\Delta_n;L))$, where $Z^1(-)$ is the functor of $1$-cocycles. A direct inspection shows that under the Dold-Kan correspondence (cf. \cite{Weibel}) $\Del_\infty(L)$ goes into the $0$-truncation of the complex $(L[1],q_1)$
	\[\xymatrix{\cdots\ar[r]^-{q_1}&L^{-1}\ar[r]^-{q_1}&L^{0}\ar[r]^-{q_1}&Z^1(L)}.\]
	In particular, by \cite{Weibel}, Theorem 8.4.1, we have that $\pi_i(\Del_\infty(L))\cong H^{1-i}(L)$ for all $i\geq0$ and (when $i\ge1$) all base points $x\in\MC(L)=Z^1(L)$.\end{remark}

\begin{definition} Given a complete $L_\infty$ algebra $L$, we denote by $\sMC(L)$ the set $\sMC(L):=\pi_0(\Del_\infty(L))$ of connected components of  $\Del_\infty(L)$.\end{definition}

\begin{proposition}\label{prop:delignevsexactsequences}  Given $0\to K\to L\to M\to 0$ a central extension of complete $L_\infty$ algebras, there is a right principal action of the simplicial abelian group $\Del_\infty(K)$ on the simplicial set $\Del_\infty(L)$. Moreover, there is an obstruction map $o:\sMC(M)\to H^2(K)$ such that the kernel of the obstruction coincides with the image of $\sMC(L)\to\sMC(M)$. Finally, if we denote by $\Del_\infty(M)^L$ the Kan subcomplex of $\Del_\infty(M)$ consisting of connected components in $\op{Ker}\,o$, then $\Del_\infty(L)\to \Del_\infty(M)^L$ is isomorphic to the principal fibration associated to the action of $\Del_\infty(K)$ on $\Del_\infty(L)$, as in \cite[\S 18]{May}.\end{proposition}

\begin{proof} We claim that a Maurer-Cartan cochain $\alpha\in\Del_\infty(M)_n$ lifts to $\beta\in\Del_\infty(L)_n$ if and only if, for some (and then for all) $i=0,\ldots,n$, its evaluation $x:=e_i^\ast(\alpha)\in\MC(M)$ at the $i$-th vertex $e_i:\Delta_0\to\Delta_n$ lifts to a Maurer-Cartan element $y\in\MC(L)$. In fact, we have a morphism of central extensions of complete $L_\infty$ algebras
	\[ \xymatrix{0\ar[r]& C^*(\Delta_n; K)\ar[r]\ar[d]_{e_i^\ast}& C^*(\Delta_n;L)\ar[r]\ar[d]_{e_i^\ast}&
		C^*(\Delta_n;M)\ar[r]\ar[d]_{e_i^\ast} & 0 \\ 0\ar[r] & K\ar[r] & L\ar[r] & M\ar[r] & 0} \]
and by Remark \ref{rem:naturalityofobstructions} $H(e^\ast_i)(o(\alpha))=o(e_i^\ast(\alpha))=o(x)$. Since $H(e_i^\ast):H^2(C^*(\Delta_n;K))\to H^2(K)$ is an isomorphism (which, being an inverse to $H(\pi^\ast):H^2(K)\to H^2(C^*(\Delta_n;K))$, $\pi$ the final morphism, doesn't depend on $i$), the claim follows. By Proposition \ref{prop:obstructionforMC}, the abelian group $\Del_\infty(K)_n=Z^1(C^*(\Delta_n;K))$ acts on the right on the set of Maurer-Cartan liftings of $\alpha$, when this is not empty. 

The above shows that the obstruction map $o:\MC(M)\to H^2(K)$ from Proposition \ref{prop:obstructionforMC} factors through the projection $\MC(M)\to\sMC(M)$, and the resulting $o:\sMC(M)\to H^2(K)$ has the required properties. Now the rest of the proposition follows easily.\end{proof}

We denote by $\underline{\mathbf{SSet}}(-,-):\mathbf{SSet}^{op}\times\mathbf{SSet}\to\mathbf{SSet}$ the simplicial mapping space functor, sending simplicial sets $X,Y$ to the simplicial set $\underline{\mathbf{SSet}}(X,Y)_n=\mathbf{SSet}(\Delta_n\times X,Y)$. The reader should compare the following theorem with \cite[Theorem 5.5]{BerglundLie} and \cite[Theorem~2.20]{BrSz}: in particular, notice that we are not putting any restriction on the simplicial set $X$ and the complete $L_\infty$ algebra $L$. We also remark that our proof is different from the ones in the above references. Here and in what follows, when we talk about weak equivalences between simplcial sets we are referring to the standard model category structure on $\mathbf{SSet}$, see \cite{BouKan,Jardine,HirMC}.

\begin{theorem}\label{th:Delignevsmappingspaces} There is a natural weak equivalence $\Del_\infty(C^*(-;-))\xrightarrow{\sim}\underline{\mathbf{SSet}}(-,\Del_\infty(-))$ of functors $\mathbf{SSet}^{op}\times\mathbf{\widehat{L_\infty}}\to\mathbf{SSet}$.\end{theorem}

\begin{proof} For a complete $L_\infty$ algebra $L$, a simplicial set $X$ and an integer $n\geq0$, we define the required $\Del_\infty(C^*(X;L))_n\xrightarrow{}\underline{\mathbf{SSet}}(X,\Del_\infty(L))_n$ as the following composition 
	\[\Del_{\infty}(C^*(X;L))_n=\MC(C^*(\Delta_n;C^*(X;L)))\xrightarrow{\MC(F)}\]
	\[\xrightarrow{\MC(F)}\MC(\widehat{\Omega}^*(\Delta_n;\widehat{\Omega}^*(X;L)))\xrightarrow{\MC(p_1^\ast\wedge p_2^\ast)}\MC(\widehat{\Omega}^*(\Delta_n\times X;L))\xrightarrow{\MC(G)}\]
	\[\xrightarrow{\MC(G)}\MC(C^*(\Delta_n\times X;L))\xrightarrow{\cong}\mathbf{SSet}(\Delta_n\times X,\Del_\infty(L))=\underline{\mathbf{SSet}}(X,\Del_\infty(L))_n, \]
	where the $L_\infty$ morphisms
	\[ F:C^*(\Delta_n;C^*(X;L))\to\widehat{\Omega}^*(\Delta_n;C^*(X;L))\to\widehat{\Omega}^*(\Delta_n;\widehat{\Omega}^*(X; L)),\]
	and $G:\widehat{\Omega}^*(\Delta_n\times X;L)\to C^*(\Delta_n\times X; L)$ are induced via homotopy transfer along Dupont's contractions, $p_1$ and $p_2$ are the projections of $\Delta_n\times X$ onto the first and second factor respectively, and finally $p_1^\ast\wedge p_2^\ast:\widehat{\Omega}^*(\Delta_n;\widehat{\Omega}^*(X; L))\to\widehat{\Omega}^*(\Delta_n\times X; L)$ is induced by
	\[\Omega^*(\Delta_n)\ten\Omega^*(X)\xrightarrow{p_1^\ast\ten p_2^\ast}\Omega^*(\Delta_n\times X)^{\ten2}\xrightarrow{\wedge}\Omega^*(\Delta_n\times X),\]
	where $\wedge:\Omega^*(\Delta_n\times X)^{\ten2}\xrightarrow{}\Omega^*(\Delta_n\times X)$ is the wedge product of forms. The fact that this defines a morphism of simplicial sets, as well as the fact that this is natural in $X$ and $L$, are both consequences of Lemma \ref{lem:transfer}.
	
	Having defined the natural $\Del_\infty(C^*(X;L))\xrightarrow{}\underline{\mathbf{SSet}}(X,\Del_\infty(L))$, we have to check that this is a weak equivalence. The rest of the proof shall depend on some general results from model category theory, for which we refer to \cite{HirMC}.
	Pull back from the terminal morphism $\pi:X\to\Delta_0$ induces
	\[\Del_\infty(L)=\Del_\infty(C^*(\Delta_0;L))\to\Del_\infty(C^*(X;L))\qquad\mbox{and}\] \[\Del_\infty(L)=\underline{\mathbf{SSet}}(\Delta_0,\Del_\infty(L))\to\underline{\mathbf{SSet}}(X,\Del_\infty(L)),\]
	and the following diagram is commutative
	\[\xymatrix{\Del_\infty(C^*(X;L))\ar[rr]&&\underline{\mathbf{SSet}}(X,\Del_\infty(L))\\
		& \Del_\infty(L)\ar[ul]\ar[ur] & }\]
	When $X=\Delta_m$ is a simplex, since $\pi:\Delta_m\to\Delta_0$ is a weak equivalence and $\Del_\infty(L)$ is a Kan complex, it follows from \cite[Corollary 9.3.3 (2)]{HirMC} that $\Del_\infty(L)\to\underline{\mathbf{SSet}}(\Delta_m,\Del_\infty(L))$ is a weak equivalence. We claim that $\Del_\infty(L)\to\Del_\infty(C^*(\Delta_m;L))$ is a weak equivalence for all $m\geq0$, thus, by two out of three, so is $\Del_\infty(C^*(\Delta_m;L))\xrightarrow{}\underline{\mathbf{SSet}}(\Delta_m,\Del_\infty(L))$.
	
	When $L$ is abelian so is $C^*(\Delta_m;L)$, for all $m\geq0$, and the claim follows from Remark \ref{ex:deligneabeliancase}. We suppose inductively that the claim has been proved for the nilpotent $L_\infty$ algebra $L/F^pL$. We notice that in the diagram
	\[\xymatrix{0\ar[r]&C^*(\Delta_m;F^pL/F^{p+1}L)\ar[r]&C^*(\Delta_m;L/F^{p+1}L)\ar[r]& C^*(\Delta_m;L/F^pL)\ar[r]&0\\
		0\ar[r]&F^{p}L/F^{p+1}L\ar[r]\ar[u]^{\pi^\ast}&L/F^{p+1}L\ar[r]\ar[u]^{\pi^\ast}&L/F^pL\ar[r]\ar[u]^{\pi^\ast}&0}\]
	the rows are central extensions of complete $L_\infty$ algebras. We will use Proposition \ref{prop:delignevsexactsequences}: recall from its claim the definition of $\Del_\infty(L/F^{p}L)^{L/F^{p+1}L}$, $\Del_\infty(C^*(\Delta_m;L/F^{p}L))^{C^*(\Delta_m;L/F^{p+1}L)}$. By the inductive hypothesis $\Del_\infty(L/F^pL)\to\Del_\infty(C^\ast(\Delta_m;L/F^pL))$ is a weak equivalence, which restricts to a weak equivalence $\Del_\infty(L/F^{p}L)^{L/F^{p+1}L}\to\Del_\infty(C^*(\Delta_m;L/F^{p}L))^{C^*(\Delta_m;L/F^{p+1}L)}$ by naturality of the obstructions: but then $\Del_\infty(L/F^{p+1}L)\to\Del_\infty( C^*(\Delta_m;L/F^{p+1}L))$ is a morphism of principal fibrations inducing weak equivalences between the bases and the fibres, hence a weak equivalence, concluding the inductive step. The claim for $L=\ilim\,L/F^pL$  follows since \[\ilim\Del_\infty(L/F^pL)=\Del_\infty(L)\to\Del_\infty(C^*(\Delta_m;L))=
	\ilim \Del_\infty(C^*(\Delta_m;L/F^pL)) \]
	is the limit of a weak equivalence between fibrant towers of simplicial sets, cf. \cite[Proposition 15.10.12 (2)]{HirMC}.
	
	Finally, denoting by $(\Delta X)^{op}$ the opposite of the small category of simplices of $X$, we claim that the morphism $\Del_\infty(C^*(X;L))\to\underline{\mathbf{SSet}}(X,\Del_\infty(L))$ is the limit 
	\[\ilim_{\{ \sigma:\Delta_n\to X\}\in(\Delta X)^{op}}\,\Del_\infty(C^*(\Delta_{n};V))\to \ilim_{\{ \sigma:\Delta_n\to X\}\in(\Delta X)^{op}}\,\underline{\mathbf{SSet}}(\Delta_{n},\Del_\infty(L))\]
	of a weak equivalence between Reedy fibrant diagrams of simplicial sets over $(\Delta X)^{op}$. Since the Reedy category $(\Delta X)^{op}$ has cofibrant constants \cite[Proposition 15.10.4 (2)]{HirMC}, the limit functor $\ilim:\mathbf{SSet}^{(\Delta X)^{op}}\to\mathbf{SSet}$ is a right Quillen functor \cite[Theorem 15.10.8 (1)]{HirMC}, and in particular it sends weak equivalences between fibrant objects to weak equivalences \cite[Proposition 8.5.7 (2)]{HirMC}: this concludes the proof that  $\Del_\infty(C^*(X;L))\to\underline{\mathbf{SSet}}(X,\Del_\infty(L))$ is a weak equivalence.
	
    To check that $(\Delta X)^{op}\to\mathbf{SSet}:\{ \sigma:\Delta_n\to X\} \to \Del_\infty(C^*(\Delta_n;L))$ is a Reedy fibrant diagram in the Reedy model category $\mathbf{SSet}^{(\Delta X)^{op}}$, we use Lemma \ref{lem:Delvslimits} to deduce that the matching morphism (cf. \cite[Definition 15.2.5]{HirMC}) at a simplex $\sigma:\Delta_n\to X$ identifies with the restriction 
	\[ \Del_\infty(C^*(\Delta_{n};L))\to\Del_\infty(C^*(\partial\Delta_{n};L)),\]
	which is a Kan fibration according to Theorem \ref{th:surjective=Kanfibration}. 
	
	Similarly, to check that $(\Delta X)^{op}\to\mathbf{SSet}:\{ \sigma:\Delta_n\to X\} \to\underline{\mathbf{SSet}}(\Delta_{n},\Del_\infty(L))$ is a Reedy fibrant diagram in $\mathbf{SSet}^{(\Delta X)^{op}}$, we observe that the matching morphism at $\sigma:\Delta_n\to X$ identifies with the restriction 
	\[ \underline{\mathbf{SSet}}(\Delta_{n},\Del_\infty(L))\to\underline{\mathbf{SSet}}(\partial\Delta_{n},\Del_\infty(L)), \]
	which is a Kan fibration according to \cite[Proposition 9.3.1 (1)]{HirMC}. \end{proof}

\section{Descent of Deligne-Getzler $\infty$-groupoids}

\newcommand{\scs}{\underrightarrow{\Delta}}

\subsection{Totalization and homotopy limits} The aim of this subsection is to review the construction of totalization and homotopy limit functors for complete $L_\infty$ algebras. Our constructions are straightforward adaptations of the standard ones for simplicial sets (see e.g. \cite{BouKan,Jardine}), where, in light of Theorem \ref{th:Delignevsmappingspaces}, we replace each occurence of the simplicial mapping space functor $\underline{\mathbf{SSet}}(X,-):\mathbf{SSet}\to\mathbf{SSet}$ by the corresponding functor of non-degenerate cochains $C^*(X;-):\widehat{\mathbf{L_\infty}}\to\widehat{\mathbf{L_\infty}}$.

 As usual, we denote by $\Delta$ the \emph{cosimplicial indexing category}, whose objects are the finite ordinals $\underline{n}=\{ 0<\cdots< n\}$ and whose morphism are the order preserving maps $\underline{m}\to\underline{n}$. Given a category $\mathbf{C}$, the category of \emph{cosimplicial objects} in $\mathbf{C}$ is the category $\mathbf{C}^{\Delta}$, whose objects are the functors $\Delta\to\mathbf{C}$ and arrows the natural transformations.

\begin{definition} Given a cosimplicial complete $L_\infty$ algebra $L_\bullet\in \widehat{\mathbf{L_\infty}}^\Delta$, with cofaces $\partial^j:L_{n-1}\to L_{n}$ and codegeneracies $s^j:L_{n+1}\to L_{n}$, $j=0,\ldots,n$,  the $n$-th \emph{matching space} is the complete $L_\infty$ algebra	\[ M_n(L_\bullet) =\{ (x_0,\ldots, x_n)\in L_n\times\cdots\times L_n\,\,\operatorname{s.t.}\,\, s^i(x_j) = s^{j-1}(x_i),\,\,\forall\,\,0\leq i<j\leq n \}.\]
	There is a natural morpshim $L_{n+1}\to M_n(L_\bullet):x\to(s^0(x),\ldots,s^n(x))$. 
\end{definition}

\begin{remark}\label{rem:reedyfibrant} Every cosimplicial complete $L_\infty$ algebra $L_\bullet$ is \emph{Reedy fibrant}, in the sense that the natural $L_{n+1}\to M_n(L_\bullet)$ is surjective for all $n\geq0$. The proof of this claim is standard, see \cite[p. 276]{BouKan} and \cite[Theorem 17.1]{May}. Given $(x_0,\ldots,x_n)\in M_n(L_\bullet)$, we put $y_n=\partial^n(x_n)$, and for $0\leq r<n$ we put inductively $y_r = \partial^r(x_r-s^r(y_{r+1})) + y_{r+1}$. Then an easy computation, using the cosimplicial relations, shows that $s^i(y_r) = x_i$ for all $r\leq i\leq n$, and in particular $y_0\in L_{n+1}$ maps onto $(x_0,\ldots, x_n)$ under the matching morphism $L_{n+1}\to M_n(L_\bullet)$. 
\end{remark}

Our first objective is to define the totalization functor $\Tot(-):\widehat{\mathbf{L_\infty}}^\Delta\to\widehat{\mathbf{L_\infty}}$. In the following definition, given a cosimplcial complete $L_\infty$ algebra, we denote by \[ \partial^j_\ast:C^*(\Delta_{n-1};L_{n-1})\to C^*(\Delta_{n-1};L_{n}),\qquad s^j_\ast:C^*(\Delta_{n+1};L_{n+1})\to C^*(\Delta_{n+1};L_n)\] 
$j=0,\ldots,n$, the push-forwards by the cofaces and the codegeneracies of $L_\bullet$, and by \[ \delta_j^*:C^*(\Delta_n;L_n)\to C^*(\Delta_{n-1};L_n),\qquad \sigma^*_j:C^*(\Delta_n;L_n)\to C^*(\Delta_{n+1};L_n)\] 
$j=0,\ldots,n$, the pull-backs by the cofaces and the codegeneracies of the standard cosimplicial simplex $\Delta_\bullet\in\mathbf{SSet}^\Delta$.

\begin{definition}\label{def:cos} Given a cosimplcial complete $L_\infty$ algebra $L_\bullet\in\widehat{\mathbf{L_\infty}}^\Delta$, its \emph{totalization} $\Tot(L_\bullet)$ is the complete $L_\infty$ algebra 
\[  \Tot(L_\bullet)=\left\{ (\alpha_0,\ldots,\alpha_n,\ldots)\in\prod_{n\geq0}C^*(\Delta_n;L_n)\,\,\operatorname{s.t.}\,\,\de^j_\ast(\alpha_{n-1}) =\delta^*_j(\alpha_n),\,s^j_*(\alpha_{n+1})=\sigma^*_j(\alpha_n) \right\}.  \] 
In other words, $\Tot(L_\bullet)$ is the universal object in $\widehat{\mathbf{L}_\infty}$ equipped with maps $\Tot(L_\bullet)\to C^*(\Delta_n;L_n)$, for all $n\geq0$, such that for every morphism $\underline{i}\to\underline{j}$ in $\Delta$, the corresponding diagram
\begin{equation}\label{eq:diagramtot}
\xymatrix{\Tot(L_\bullet)\ar[r]\ar[d]&C^*(\Delta_j;L_j)\ar[d]\\C^*(\Delta_i;L_i)\ar[r]&
		C^*(\Delta_i;L_j)}\end{equation}
	is commutative. We can similarly define $L_\infty$ subalgebras $\Tot_k(L_\bullet)\subset\prod_{0\leq n\leq k}C^*(\Delta_n;L_n)$, $k\geq0$, by an analog universal property, where we require commutativity of the above diagram only for those morphisms $\underline{i}\to\underline{j}$ in $\Delta$ where $i,j\leq k$. 
	
	More formally, following Hinich \cite{Hinichdescent} we introduce a category $\sM$ (resp.: $\sM_{\leq k}$) whose objects are the arrows $\underline{i}\to\underline{j}$ in $\Delta$ (resp.: with $i,j\leq k$) and whose arrows $\left\{\underline{i}\rh\underline{j}\right\}\rh\left\{\underline{i'}\rh\underline{j'}\right\}$ are the factorizations $\left\{\underline{i'}\rh\underline{j'}\right\}=\left\{\underline{i'}\rh\underline{i}\rh\underline{j}\rh\underline{j'}\right\}$ in $\Delta$: then the above says that $\Tot(L_\bullet)$ (resp.: $\Tot_k(L_\bullet)$) is the limit of the functor $\sM\rh\widehat{\mathbf{L_\infty}}:\left\{\underline{i}\rh\underline{j}\right\}\rh C^*(\Delta_i;L_j)$ (resp.: of the restricition of this functor to $\sM_{\leq k}$). 
\end{definition}\begin{remark}\label{rem:cos} There is a natural isomorphism $\Tot(L_\bullet)=\ilim\Tot_k(L_\bullet)$. Moreover, obviously $\Tot_0(L_\bullet)=L_0$, and for every $k\geq1$ the natural projection $\Tot_k(L_\bullet)\to\Tot_{k-1}(L_\bullet)$ fits into  the following cartesian square in $\widehat{\mathbf{L_\infty}}$
	\begin{equation}\label{eq:cartesian}\xymatrix{\Tot_k(L_\bullet)\ar[r]\ar[d]& C^*(\Delta_k;L_k)\ar[d]\\\Tot_{k-1}(L_\bullet)\ar[r]&
		N_{k-1}} \end{equation}
where the complete $L_\infty$ algebra $N_{k-1}$ is the fiber product
\[ N_{k-1}:= C^*(\de\Delta_k;L_k)\times_{C^*(\de\Delta_k;M_{k-1}(L_\bullet))}C^*(\Delta_k;M_{k-1}(L_\bullet)).   \]
We notice that Remark \ref{rem:reedyfibrant} implies that the vertical arrows in the above cartesian diagram \eqref{eq:cartesian} are surjections.\end{remark}
In some situations codegeneracies don't play a significant role, and it is convenient to drop them altogether: this leads to the notion of semicosimplicial object in a category. We shall denote by $\scs\subset\Delta$ the \emph{semicosimplicial indexing category}, whose objects are once again the finite ordinals $\underline{n}$, but where the arrows are the \emph{injective} order preserving maps $\underline{m}\to\underline{n}$. Given a category $\mathbf{C}$, the category of \emph{semicosimplicial objects} in $\mathbf{C}$ is the category $\mathbf{C}^{\scs}$ with objects the functors $\scs\to\mathbf{C}$ and arrows the natural transformations. In other words, a semicosimplicial object in $\mathbf{C}$ is roughly the same as a cosimplicial object without the codegeneracies.

In the following definition, the morphisms $\partial^j_\ast:C^*(\Delta_{n-1};L_{n-1})\to C^*(\Delta_{n-1};L_{n})$ and $\delta_j^*:C^*(\Delta_{n};L_{n})\to C^*(\Delta_{n-1};L_{n})$ have the same meaning as before

\begin{definition}\label{def:scs} Given a semicosimplcial complete $L_\infty$ algebra $L_\bullet\in\widehat{\mathbf{L_\infty}}^{\scs}$, its \emph{totalization} $\Tot(L_\bullet)$ is the complete $L_\infty$ algebra 
	\[  \Tot(L_\bullet)=\left\{ (\alpha_0,\ldots,\alpha_n,\ldots)\in\prod_{n\geq0}C^*(\Delta_n;L_n)\,\,\operatorname{s.t.}\,\,\de^j_\ast(\alpha_{n-1}) =\delta^*_j(\alpha_n) \right\}.  \] 
	The complete $L_\infty$ algebra $\Tot(L_\bullet)$ satisfies an universal property analog to the previous one (where we require commutativity of the diagram \eqref{eq:diagramtot} only for those morphisms $\underline{i}\to\underline{j}$ in $\scs\subset\Delta$). The partial totalizations $\Tot_k(L_\bullet)$, $k\geq0$, are defined similarly as before. 
	
	Introducing the category $\underrightarrow{\sM}$ (resp.: $\underrightarrow{\sM}_{\leq k}$) whose objects are the arrows $\underline{i}\to\underline{j}$ in $\scs$ (resp.: with $i,j\leq k$) and whose arrows $\left\{\underline{i}\rh\underline{j}\right\}\rh\left\{\underline{i'}\rh\underline{j'}\right\}$ are the factorizations $\left\{\underline{i'}\rh\underline{j'}\right\}=\left\{\underline{i'}\rh\underline{i}\rh\underline{j}\rh\underline{j'}\right\}$ in $\scs$, then $\Tot(L_\bullet)$ is the limit of $\underrightarrow{\sM}\rh\widehat{\mathbf{L_\infty}}:\left\{\underline{i}\rh\underline{j}\right\}\rh C^*(\Delta_i;L_j)$. Similarly, $\Tot_k(L_\bullet)$ is the limit of the restrictions of the above functor to $\underrightarrow{\sM}_{\leq k}\subset\underrightarrow{\sM}$.
\end{definition}\begin{remark}\label{rem:scs} Once again, there is a natural isomorphism $\Tot(L_\bullet)=\ilim\Tot_k(L_\bullet)$. Furthermore, $\Tot_0(L_\bullet)=L_0$, and for every $k\geq1$ the natural projection $\Tot_k(L_\bullet)\to\Tot_{k-1}(L_\bullet)$ fits into  the cartesian square in $\widehat{\mathbf{L_\infty}}$
\begin{equation}\label{eq:cartesianscs}\xymatrix{\Tot_k(L_\bullet)\ar[r]\ar[d]& C^*(\Delta_k;L_k)\ar[d]\\\Tot_{k-1}(L_\bullet)\ar[r]&
	C^*(\de\Delta_k;L_k)} \end{equation}
\end{remark}

We finally come to the definition of homotopy limits in $\widehat{\mathbf{L}_\infty}$. 

\begin{definition}\label{def:replacement} Given a small category $\mathbf{S}$ and an $\mathbf{S}$-diagram $F:\mathbf{S}\to\widehat{\mathbf{L}_\infty}$ of complete $L_\infty$ algebras, its \emph{cosimplicial replacement} is the cosimplicial complete $L_\infty$ algebra
 \[ \Pi(L_\bullet):\quad \xymatrix{ \Pi(L_\bullet)_0
 	\ar@<2pt>[r]\ar@<-2pt>[r] & \Pi(L_\bullet)_1    \ar@<4pt>[r] \ar[r] \ar@<-4pt>[r] & \Pi(L_\bullet)_2 \ar@<6pt>[r] \ar@<2pt>[r] \ar@<-2pt>[r] \ar@<-6pt>[r]&\cdots}  \]
(where we omitted the codegeneracies for notational simplicity) whose complete $L_\infty$ algebra of $n$-simplices is
\[ \Pi(L_\bullet)_n = \prod_{i_0\xrightarrow{}\cdots\xrightarrow{} i_n} F(i_n) \]
where the product runs over the set of $n$-uplets of composable arrows in $\mathbf{S}$. 

The cofaces and the codegeneracies are defined as follows. The composition of the $j$-th coface $\partial^j:\Pi(F)_{n-1}\to\Pi(F)_n$, $j=1,\ldots,{n-1}$ (resp.: $j=0$, $j=n$), and the projection onto the factor $F(i_n)$ indexed by $i_0\xrightarrow{\phi_1}\cdots\xrightarrow{\phi_n} i_n$ coincides with the projection onto the factor $F(i_n)$ (resp.: $F(i_n)$, $F(i_{n-1})$) indexed by $i_0\xrightarrow{\phi_1}\cdots i_{j-1}\xrightarrow{\phi_{j+1}\circ\phi_j}i_{j+1}\cdots\xrightarrow{\phi_n} i_n$ (resp.: indexed by $i_1\xrightarrow{\phi_2}\cdots\xrightarrow{\phi_{n}} i_{n}$, $i_0\xrightarrow{\phi_1}\cdots\xrightarrow{\phi_{n-1}} i_{n-1}$), followed by the identity $\id_{F(i_n)}$ (resp.: followed by $\id_{F(i_n)}$, $F(\phi_n)$). The composition of the $j$-th codegeneracy $s^j:\Pi(F)_{n+1}\to\Pi(F)_n$, $j=0,\ldots,n$, and the projection onto onto the factor $F(i_n)$ indexed by $i_0\xrightarrow{\phi_1}\cdots\xrightarrow{\phi_n} i_n$ coincides with the projection onto the factor $F(i_n)$ indexed by $i_0\xrightarrow{\phi_1}\cdots i_{j}\xrightarrow{\id_{i_j}}i_{j}\cdots\xrightarrow{\phi_n} i_n$  followed by the identity of $F(i_n)$.
\end{definition}
\begin{definition} Given a small category $\mathbf{S}$, the \emph{homotopy limit} functor $\underleftarrow{\operatorname{holim}}(-):\widehat{\mathbf{L}_\infty}^{\mathbf{S}}\to\widehat{\mathbf{L}_\infty}$ is the composition of the cosimplicial replacement functor $\Pi(-):\widehat{\mathbf{L}_\infty}^{\mathbf{S}}\to\widehat{\mathbf{L}_\infty}^{\Delta}$ and the totalization functor $\Tot(-):\widehat{\mathbf{L}_\infty}^\Delta\to\widehat{\mathbf{L}_\infty}$.
\end{definition}
\begin{remark} As already remarked, we can repeat the previous discussion almost verbatim, replacing every instance of the word complete $L_\infty$ algebra by simplicial set, and every instance of the functor $C^*(X;-):\widehat{\mathbf{L}_\infty}\to\widehat{\mathbf{L}_\infty}$ by the corresponding mapping space functor $\underline{\mathbf{SSet}}(X,-):\mathbf{SSet}\to\mathbf{SSet}$, and we recover the usual constructions of totalization and homotopy limits in the context of simplicial sets, cf. \cite[Chapters X, XI]{BouKan} and \cite[Chapter VII]{Jardine}.
\end{remark}
\subsection{Descent of Deligne-Getzler $\infty$-groupoids}\label{subs:descent} In this subsection, we prove the compatibility, up to homotopy, between $\Del_\infty(-)$ and the totalization and homotopy limit functors introduced before.

We begin by showing the compatibility between $\Del_\infty(-)$ and $\Tot(-)$. As a preliminary result in this direction, we observe that given either a cosimplicial or a semicosimplicial complete $L_\infty$ algebra $L_\bullet$, the two simplicial sets $\Del_\infty(\Tot(L_\bullet))$ and $\Tot(\Del_\infty(L_\bullet))$ have the same set of vertices.

\begin{proposition}\label{prop:TotvsMC} There are natural isomorphism $\MC(\Tot(-))\xrightarrow{\cong}\Tot(\Del_\infty(-))_0$ of functors $\widehat{\mathbf{L}_\infty}^{\Delta}\rh\mathbf{Set}$ and $\widehat{\mathbf{L}_\infty}^{\scs}\rh\mathbf{Set}$.\end{proposition}
\begin{proof} To fix the ideas, we consider the cosimplicial case. Since $\MC(-)$ commutes with small limits
	\[\MC(\Tot(L_\bullet)) = \ilim_{\sM}\,\MC(C^*(\Delta_i;L_j)) = \ilim_{\sM}\,\mathbf{SSet}(\Delta_i,\Del_\infty(L_j)) =
	\Tot(\Del_\infty(L_\bullet))_0.\]
The semicosimplicial case is proved in the same way, replacing the category $\sM$ by the one $\underrightarrow{\sM}$ in the above chain of isomorphisms.
\end{proof}

\begin{theorem}\label{th:descent} There are natural weak equivalences $\Del_\infty(\Tot(-))\xrightarrow{\sim}\Tot(\Del_\infty(-))$ of functors $\widehat{\mathbf{L}_\infty}^{\Delta}\rh\mathbf{SSet}$ and $\widehat{\mathbf{L}_\infty}^{\scs}\rh\mathbf{SSet}$.\end{theorem}

\begin{proof} In both the cosimplicial and the semicosimplicial cases, we have morphisms \[ \Del_\infty(\Tot(L_\bullet))\rh\Del_\infty(C^*(\Delta_i;L_i))\rh\underline{\mathbf{SSet}}(\Delta_i,\Del_\infty(L_i)),\quad i\geq0,\] 
given by Theorem \ref{th:Delignevsmappingspaces}. For each arrow $\{i\rh j\}$ in $\Delta$ (resp.: $\underrightarrow{\Delta}$) in the induced diagram
\[\xymatrix{\Del_\infty(\Tot(L_\bullet))\ar[d]\ar[r]&\Del_\infty(C^*(\Delta_j;L_j))\ar[r]\ar[d]&\underline{\mathbf{SSet}}(\Delta_j,\Del_\infty(L_j))\ar[dd]\\
		\Del_\infty(C^*(\Delta_i;L_i))\ar[d]\ar[r]&\Del_\infty(C^*(\Delta_i;L_j))\ar[rd]\\
		\underline{\mathbf{SSet}}(\Delta_i,\Del_\infty(L_i))\ar[rr]&&\underline{\mathbf{SSet}}(\Delta_i,\Del_\infty(L_j))}\]
	the inner squares are commutative, thus the outer square is commutative as well, and there is induced a unique natural $\Del_\infty(\Tot(L_\bullet))\rh\Tot(\Del_\infty(L_\bullet))$ making the diagram
	\[\xymatrix{\Del_\infty(\Tot(L_\bullet))\ar[d]\ar[r]&\Del_\infty(C^*(\Delta_i;L_i))\ar[d]\\
		\Tot(\Del_\infty(L_\bullet))\ar[r]&\underline{\mathbf{SSet}}(\Delta_i,\Del_\infty(L_i))}\]
	commutative for all $i\geq0$. For the same reason, for each $k\geq0$ it is defined a natural transformation $\Del_\infty(\Tot_{k}(-))\rh\Tot_{ k}(\Del_\infty(-))$.
	
	To prove that $\Del_\infty(\Tot(L_\bullet))\xrightarrow{}\Tot(\Del_\infty(L_\bullet))$ is a weak equivalence, we shall consider the semicosimplicial case first. We will prove inductively that $\Del_\infty(\Tot_{ k}(L_\bullet))\xrightarrow{}\Tot_{ k}(\Del_\infty(L_\bullet))$ is a weak equivalence, the case $k=0$ being obvious. To continue the induction we look at the commutative diagram (recall Remark \ref{rem:scs})
	\[ \xymatrix{\Del_\infty(\Tot_{ k-1}(L_\bullet))\ar[r]\ar[d]&\Del_\infty(C^*(\de\Delta_k;L_k))\ar[d]&\Del_\infty(C^*(\Delta_k;L_k))\ar[d]\ar[l]\\
		\Tot_{ k-1}(\Del_\infty(L_\bullet))\ar[r]&\underline{\mathbf{SSet}}(\de\Delta_k,\Del_\infty(L_k))&\underline{\mathbf{SSet}}(\Delta_k,\Del_\infty(L_k))\ar[l]} \]
	where all spaces are Kan complexes. As the left pointing arrows are Kan fibrations, the top one by Theorem \ref{th:surjective=Kanfibration} and the bottom one by \cite[Proposition 9.3.1 (1)]{HirMC}, we see that the fiber products of the rows
	\[\Del_\infty(\Tot_{ k}(L_\bullet))=\Del_\infty(\Tot_{ k-1}(L_\bullet))\times_{\Del_\infty(C^*(\de\Delta_k;L_k))}\Del_\infty(C^*(\de\Delta_k;L_k))\]
	\[\Tot_{ k}(\Del_\infty(L_\bullet))=\Tot_{ k-1}(\Del_\infty(L_\bullet))\times_{\underline{\mathbf{SSet}}(\de\Delta_k,\Del_\infty(L_k))}\underline{\mathbf{SSet}}(\Delta_k,\Del_\infty(L_k))\]
	are also homotopy fiber products (cf. \cite[Remark A.2.4.5]{LurieHTT}). As the vertical arrows in the diagram are weak equivalences, by the inductive hypothesis and Theorem \ref{th:Delignevsmappingspaces}, this implies that also $\Del_\infty(\Tot_{ k}(L_\bullet))\rh\Tot_{   k}(\Del_\infty(L_\bullet))$ is a weak equivalence, and concludes the inductive step. 
	
Finally, for all $k\geq1$ the projections \[ \Del_\infty(\Tot_{ k}(L_\bullet))\rh\Del_\infty(\Tot_{ k-1}(L_\bullet))\qquad\mbox{and}\qquad\Tot_{ k}(\Del_\infty(L_\bullet))\rh\Tot_{ k-1}(\Del_\infty(L_\bullet))\] 
are the pull-backs of Kan fibrations, hence Kan fibrations themselves. It follows that \[\ilim_k \Del_\infty(\Tot_k(L_\bullet))=\Del_\infty(\Tot(L_\bullet))\to\Tot(\Del_\infty(L_\bullet))=\ilim_k\Tot_k(\Del_\infty(L_\bullet))\] 
is the limit of a weak equivalence between fibrant towers of simplicial sets, cf. \cite[Proposition 15.10.12 (2)]{HirMC}, hence a weak equivalence.

Next we consider the cosimplcial case. We shall follow the same inductive argument as before. Assume we have shown that   $\Del_\infty(\Tot_{k-1}(L_\bullet))\to\Tot_{k-1}(\Del_\infty(L_\bullet))$ is a weak equivalence, the case $k-1=0$ being obvious. The simplicial sets $\Del_\infty(\Tot_k(L_\bullet))$ and $\Tot_k(\Del_\infty(L_\bullet))$ fit into cartesian diagrams (cf. Remark \ref{rem:cos})
\[ \xymatrix{\Del_\infty(\Tot_k(L_\bullet))\ar[r]\ar[d]& \Del_\infty(C^*(\Delta_k;L_k))\ar[d]\\\Del_\infty(\Tot_{k-1}(L_\bullet))\ar[r]&
	\Del_\infty(N_{k-1})}\qquad \xymatrix{\Tot_k(\Del_\infty(L_\bullet))\ar[r]\ar[d]& \underline{\mathbf{SSet}}(\Delta_k,\Del_\infty(L_k))\ar[d]\\\Tot_{k-1}(\Del_{\infty}(L_\bullet))\ar[r]&
	X_{k-1}}\]
where 
\[ \Del_\infty(N_{k-1}):= \Del_\infty(C^*(\de\Delta_k;L_k))\times_{\Del_\infty(C^*(\de\Delta_k;M_{k-1}(L_\bullet)))}\Del_\infty(C^*(\Delta_k;M_{k-1}(L_\bullet)))  \]
and \[ X_{k-1}:= \underline{\mathbf{SSet}}(\de\Delta_k,\Del_\infty(L_k))\times_{\underline{\mathbf{SSet}}(\de\Delta_k,M_{k-1}(\Del_\infty(L_\bullet)))}\underline{\mathbf{SSet}}(\Delta_k,M_{k-1}(\Del_\infty(L_\bullet))). \]
Furthermore, the vertical arrows in the above cartesian diagrams are Kan fibrations. For the left hand side diagram, this follows from Remark \ref{rem:cos} and Theorem \ref{th:surjective=Kanfibration}. For the right hand side one, it follows from the fact that the cosimplicial simplicial set $\Del_\infty(L_\bullet)$ is Reedy fibrant (see \cite[Ch. X, \S 4]{BouKan}): this can be seen using again Theorem \ref{th:surjective=Kanfibration}, together with the observation that $\Del_\infty(-)$ commutes with the matching space functors $M_{k-1}(-)$,  according to Lemma \ref{lem:Delvslimits}.

The compatibility between $\Del_\infty(-)$ and $M_{k-1}(-)$ also implies the following commutative diagram, induced by Theorem \ref{th:Delignevsmappingspaces}
\[ \xymatrix{\Del_\infty(C^*(\de\Delta_k;L_k))\ar[r]\ar[d]&\Del_\infty(C^*(\de\Delta_k;M_{k-1}(L_\bullet)))\ar[d]&\Del_\infty(C^*(\Delta_k;M_{k-1}(L_\bullet)))\ar[d]\ar[l]\\
	\underline{\mathbf{SSet}}(\de\Delta_k,\Del_\infty(L_k))\ar[r]&\underline{\mathbf{SSet}}(\de\Delta_k,M_{k-1}(\Del_\infty(L_\bullet)))&\underline{\mathbf{SSet}}(\Delta_k,M_{k-1}(\Del_\infty(L_\bullet)))\ar[l]} \]
where all spaces are Kan complexes, and the left pointing arrows are Kan fibrations: hence the fiber products $\Del_\infty(N_{k-1})$ and $X_{k-1}$ are also homotopy fiber products, by \cite[Remark A.2.4.5]{LurieHTT}. Since the vertical arrows are weak equivalences, so is the induced $\Del_\infty(N_{k-1})\to X_{k-1}$. Now the rest of the proof proceeds exactly as in the semicosimplicial case.

	\end{proof}

We notice that the previous theorem, combined with the previous proposition, imply the following corollary (which slightly generalize a result from \cite{FMMscdgla}), where we denote by $\pi_{\leq1}(-):\mathbf{Kan}\rh\mathbf{Grpd}$ the functor sending a Kan complex to its fundamental groupoid.

\begin{corollary}\label{cor:fiomanmar} There are natural isomorphisms $\pi_{\leq1}(\Del_\infty(\Tot(-)))\xrightarrow{\cong}\pi_{\leq1}(\Tot(\Del_\infty(-)))$ of functors $\widehat{\mathbf{L_\infty}}^{\Delta}\rh\mathbf{Grpd}$ and  $\widehat{\mathbf{L_\infty}}^{\underrightarrow{\Delta}}\rh\mathbf{Grpd}$.\end{corollary}
\begin{proof} Since an equivalence of groupoids which is an isomorphism on the set of objects has to be an isomorphism.\end{proof}

Finally, the previous result immediately implies the compatibility, up to homotopy, between $\Del_\infty(-)$ and homotopy limits.

\begin{theorem}\label{th:hldescent} Given a small category $\mathbf{S}$, there is a natural weak equivalence
\[ \Del_\infty(\underleftarrow{\operatorname{holim}}(-))\xrightarrow{\sim} \underleftarrow{\operatorname{holim}}(\Del_\infty\circ - ) \]
of functors $\widehat{\mathbf{L}_\infty}^{\mathbf{S}}\to\mathbf{SSet}$.
\end{theorem}

\begin{proof} Let $F:\mathbf{S}\to \widehat{\mathbf{L}_\infty}$ be an $\mathbf{S}$-diagram of complete $L_\infty$ algebras, together with the induced $\mathbf{S}$-diagram  $\Del_\infty\circ F:\mathbf{S}\to \mathbf{SSet}$ of simplicial sets. Since $\Del_\infty(-)$ commutes with products, it is straightforward to check that there is a natural identification $\Del_\infty(\Pi(F))=\Pi(\Del_\infty\circ F)$ of cosimplicial simplicial sets, where $\Pi(-)$ is the cosimplicial replacement functor from the previous subsection. By the previous theorem, there is a natural weak equivalence
\[\Del_\infty(\underleftarrow{\operatorname{holim}}(F))= \Del_\infty(\Tot(\Pi(F)))\xrightarrow{\sim} \Tot(\Pi(\Del_\infty\circ F))=\underleftarrow{\operatorname{holim}}(\Del_\infty\circ F).\]

\end{proof}

\end{document}